\documentclass[a4]{amsart}

\usepackage[utf8]{inputenc}
\usepackage[english]{babel}
\usepackage{amssymb}
\usepackage{mathtools}
\usepackage[mathscr]{euscript}
\usepackage{mathrsfs}
\usepackage{stmaryrd}
\usepackage{enumitem}
\usepackage[normalem]{ulem}
\usepackage{color}
\setenumerate[1]{label=(\alph*)}
\setenumerate[2]{label=(\roman*)}
\usepackage[all]{xy}
\usepackage{tikz}
\usetikzlibrary{matrix,arrows}
\usepackage{rotating}
\usepackage{pifont}
\usepackage{xr-hyper}
\usepackage{hyperref}

\hypersetup{
  pdftitle    = {Categories of constructible sheaves},
  pdfauthor   = {Valery A.~Lunts, Olaf
    M.~Schn{\"u}rer, J{\"o}rg Sch{\"u}rmann},
  colorlinks=true}

\newtheorem{theorem}{Theorem}[section]
\newtheorem{proposition}[theorem]{Proposition}
\newtheorem{lemma}[theorem]{Lemma}
\newtheorem{corollary}[theorem]{Corollary}

\theoremstyle{definition}
\newtheorem{definition}[theorem]{Definition}

\newtheorem{para}[theorem]{}

\theoremstyle{remark}
\newtheorem{remark}[theorem]{Remark}

\DeclareMathOperator{\kernel}{Ker}

\DeclareMathOperator{\Hom}{Hom}

\DeclareMathOperator{\Ext}{Ext}

\DeclareMathOperator{\Cone}{Cone}

\DeclareMathOperator{\supp}{supp}

\DeclareMathOperator{\pt}{pt}

\DeclareMathOperator{\HH}{H}
\newcommand{\sing}{{\operatorname{sing}}}

\renewcommand{\epsilon}{{\varepsilon}}

\renewcommand{\phi}{{\varphi}}

\renewcommand{\tilde}[1]{{\widetilde{#1}}}

\newcommand{\define}[1]{{\textbf{#1}}}

\newcommand{\id}{{\operatorname{id}}}
\newcommand{\opp}{{\operatorname{op}}}

\newcommand{\bZ}{{\mathbb{Z}}}

\newcommand{\bC}{{\mathbb{C}}}
\newcommand{\bN}{{\mathbb{N}}}
\newcommand{\bP}{{\mathbb{P}}}

\DeclareMathOperator{\Sh}{Sh}

\DeclareMathOperator{\Mod}{Mod}

\newcommand{\Loc}{{\operatorname{Loc}}}
\newcommand{\Cons}{{\operatorname{Cons}}}
\newcommand{\cs}{{\operatorname{cs}}}

\newcommand{\real}{{\operatorname{real}}}

\DeclareMathOperator{\Top}{Top}

\newcommand{\bd}{{\operatorname{b}}}
\newcommand{\C}{\operatorname{{C}}}
\newcommand{\D}{{\operatorname{D}}}
\newcommand{\K}{{\operatorname{K}}}
\newcommand{\Rd}{\mathsf{R}}

\newcommand{\ol}[1]{{\overline{#1}}}

\newcommand{\ra}{\rightarrow}
\newcommand{\la}{\leftarrow}

\newcommand{\xra}[2][]{\xrightarrow[{#1}]{#2}}
\newcommand{\xla}[2][]{\xleftarrow[{#1}]{#2}}
\newcommand{\sira}{\xra{\sim}}

\newcommand{\sila}{\xla{\sim}}

\newcommand{\Ra}{\Rightarrow}
\newcommand{\xRa}[2][]{\xRightarrow[{#1}]{#2}}
\newcommand{\siRa}{\xRa{\sim}}

\newcommand{\ul}[1]{{\underline{#1}}}

\newcommand{\sptag}[1]{\href{http://stacks.math.columbia.edu/tag/#1}{#1}}

\makeatletter
\newcommand{\mylabel}[2]{#2\def\@currentlabel{#2}\label{#1}}
\makeatother

\numberwithin{equation}{section}

\title{Categories of constructible sheaves}
\author{Valery A.~Lunts, Olaf M.~Schn{\"u}rer}

\address{
  Department of Mathematics\\
  Indiana University\\
  831 East 3rd Street\\
  Bloomington, IN 47405\\
  USA.
National Research University Higher School of Economics, Moscow, Russia}

\email{vlunts@iu.edu}

\address{
  Institut f\"ur Mathematik\\
  Universit{\"a}t Paderborn\\
  Warburger Stra\ss{}e 100\\
  33098 Paderborn\\
  Germany
}

\email{olaf.schnuerer@math.uni-paderborn.de}

\begin{document}
\maketitle

\begin{abstract}
  Given a stratified topological space, we answer the question
  whether the functor from the derived category of constructible
  sheaves to the derived category of sheaves with constructible
  cohomology is an equivalence.

  We also establish basic facts on the category of
  locally constant sheaves and on the category of
  constructible sheaves.
\end{abstract}

\tableofcontents

\section{Introduction}
\label{sec:introduction}

Let $\mathcal{A}$ be an abelian category and $\mathcal{B}\subset \mathcal{A}$ its full abelian subcategory which is closed under extensions in $\mathcal{A}$. One may consider the bounded derived category $\D ^b (\mathcal{A})$ and its full triangulated subcategory $\D ^b _{\mathcal{B}} (\mathcal{A})\subset \D ^b (\mathcal{A})$ consisting of complexes with cohomology in $\mathcal{B}$. We have the obvious {\it realization} functor
\begin{equation}\label{real intr}
\real :\D ^b(\mathcal{B})\to \D ^b_{\mathcal{B}}(\mathcal{A})
\end{equation}
and it is often important to know whether $\real$ is an equivalence.
This question is too general: there are some meaningful results only for concrete examples of $\mathcal{A}$ and $\mathcal{B}$. (See \cite{LaLu} for a purely algebraic example).

Actually, perhaps the most interesting examples come from topology, where $\mathcal{A}$ and $\mathcal{B}$ are certain categories of sheaves.

For example, one may consider a topological space $X$ with a (finite) stratification $\mathcal{S}$. Take for $\mathcal{A}$ the category $\Sh(X)$ of all sheaves of abelian groups and for $\mathcal{B}$ - its full subcategory $\Cons(X,\mathcal{S})$ of sheaves, which are $\mathcal{S}$-constructible (that is whose restriction to every stratum is a local system, possibly of infinite rank). In this work we find conditions which guarantee, that the corresponding realization functor
\begin{equation}\label{real intr 2}\real :\D ^+(\Cons(X,\mathcal{S}))\to
\D ^+_{\mathcal{S}}(\Sh (X))
\end{equation}
is an equivalence. The similar question about constructible sheaves of finite rank is more subtle and there we have only partial results.

If one does not fix a stratification of the space, but rather considers  sheaves which are constructible with respect to some (for example algebraic) stratification, then the functor \eqref{real intr 2} has a better chance to be an equivalence. To our best knowledge the paper \cite{nori-constructible-sheaves} contains the most general results in this setup.

The authors are grateful to J{\"o}rg Sch{\"u}rmann for a very useful discussion of the subject.

\subsection{Conventions}
\label{sec:conventions}

All rings are assumed to be associative and unital, but not
necessarily commutative. The symbol $R$ always denotes a ring.
Modules are left modules, and $\Mod(R)$ denotes the category of
left $R$-modules.

As a general rule, a sheaf on a
topological space $X$ means a sheaf of $R$-modules. The category
of sheaves on $X$ is denoted by $\Sh(X)$. When we speak
about sheaves of sets we say this
explicitly. The same convention applies to presheaves.

If $X$ is a topological space,
$\HH^\sing_q(X,A)$ and
$\HH_\sing^q(X, A)$ denote the $q$-th
singular homology and cohomology, respectively, of $X$ with
values in an abelian group or $R$-module $A$;
$\HH_\sing(X, F)$ denotes the
sheaf cohomology of $X$ with values in a sheaf $F$ of abelian
groups or $R$-modules.

If $\mathcal{A}$ is an additive category, its category of
complexes is denoted by $\C(\mathcal{A})$, its homotopy category
by $\K(\mathcal{A})$.

Given an abelian category $\mathcal{A}$, an abelian subcategory
is a strictly full subcategory $\mathcal{B}$ which is abelian and
whose inclusion functor $\mathcal{B} \ra \mathcal{A}$ is exact.

The (unbounded) derived category of an abelian category
$\mathcal{A}$ is denoted by
$\D(\mathcal{A})$. By $\D^\bd(\mathcal{A})$,
$\D^+(\mathcal{A})$,
$\D^-(\mathcal{A})$ we denote the full
subcategories of $\D(\mathcal{A})$
of objects with bounded, bounded below, bounded above cohomology,
respectively.

If $\mathcal{B}$ is a full subcategory of
$\mathcal{A}$, we denote by $\D_\mathcal{B}(\mathcal{A})$
the full subcategory of $\D(\mathcal{A})$
of objects $E$ whose cohomology sheaves $\HH^p(E)$ are in
$\mathcal{B}$, for all $p \in \bZ$.
If $\mathcal{B}$ is a weak Serre subcategory
of $\mathcal{A}$, then
$\D_\mathcal{B}(\mathcal{A})$ is a thick triangulated subcategory
of $\D(\mathcal{A})$
(see \cite[\sptag{02MN}, \sptag{06UQ}]{stacks-project}).
The categories $\D^\bd_\mathcal{B}(\mathcal{A})$,
$\D^+_\mathcal{B}(\mathcal{A})$,
$\D^-_\mathcal{B}(\mathcal{A})$ are defined in the obvious
way.

We abbreviate $\D(R):=\D(\Mod(R))$ and $\D(X):=\D(\Sh(X))$ and
$\K(X):=\K(\Sh(X))$ and
$\C(X):=\C(\Sh(X))$.

The shift functor of a triangulated category is sometimes denoted
by $\Sigma$ and sometimes by $[1]$.


\section{Formulation of main results}
\label{sec:formulation of main res}

\subsection{Locally constant sheaves}

\begin{definition}
  Let $X$ be a topological space.
  \begin{enumerate}
  \item
    If $M$ is an $R$-module, the sheaf on $X$ associated
    to the presheaf $U \mapsto M$ is denoted by
    $\ul{M}=\ul{M}_X$ and called the \define{constant sheaf with
      stalk $M$}. Equivalently, it can be described as the sheaf
    of local sections of $M \times X \ra X$ where $M$ carries the
    discrete topology.
  \item
    A sheaf on $X$ is called \define{constant} if it is
    isomorphic to a sheaf of the form $\ul{M}_X$.
  \item
    A sheaf $F$ on $X$ is called \define{locally constant} if
    every point of $X$ has an open neighborhood $U$ such that $F|_U$ is
    constant.
  \end{enumerate}
  The category of locally constant sheaves is denoted by $\Loc(X)$. It is a full subcategory $\Sh (X)$.
\end{definition}

\begin{definition} \label{d:sc-sa}

   A topological space $X$ is {\bf simply connected} if $X$ is path connected and satisfies $\pi_1(X,x)=1$ for a/any $x
    \in X$.  We say that $X$ is {\bf acyclic} (resp. {\bf 1-acyclic}) if the sheaf cohomology $\HH ^0(X,\ul{M})=M$ and $\HH ^{>0}(X,\ul{M})=0$ (resp. $\HH ^1(X,\ul{M})=0$) for all $R$-modules $M$.
\end{definition}

\begin{definition}
  \label{d:open-locally-1}
   If (P) is a property of topological spaces, we say that a
  topological space $X$ is \define{locally (P)} if any
  neighborhood of any point $x \in X$ contains an open
  neighborhood $U$ of $x$ having property (P). Equivalently, this
  means that the topology of $X$ has a basis consisting of open
  subsets having property (P).
\end{definition}

\begin{para}
  Every (locally) contractible space is (locally) simply
  connected
  and (locally) acyclic.
\end{para}



\begin{proposition}
  \label{p:fiber-functor-1}
  Let $X$ be non-empty, path connected, locally simply connected
  topological space. Let $x \in X$. Then the following holds.

(1) There is a natural equivalence of categories
  \begin{equation*}
    \Loc(X) \sira \Mod(R \pi_1(X,x)).
  \end{equation*}
Therefore the category $\Loc (X)$ is a (complete and co-complete) Grothendieck abelian category. It has enough projectives and enough injectives.

(2) The inclusion functor $\Loc(X)\hookrightarrow \Sh(X)$ is exact, continuous and co-continuous (i.e. preserves inverse and direct limits).

(3)  Assume in addition that $X$ is locally 1-acyclic. Then $\Loc(X)$ is closed under extensions in $\Sh (X)$.
\end{proposition}

\begin{para} If $\Loc(X)$ is closed under extensions in $\Sh (X)$, then $\D_{\Loc(X)}(\Sh(X))$ is a thick triangulated subcategory of $\D(\Sh(X))$
and one can ask if the obvious triangulated functor
$\D(\Loc(X))\to \D_{\Loc(X)}(\Sh(X))$ is an equivalence.
\end{para}

\begin{theorem}
  \label{t:main-one-stratum-1}
  Let $X$ be a
  locally simply connected,
  locally acyclic
  topological space. Let $\tilde{X}$ be the disjoint union of the
  universal coverings of all path components of $X$.
  Then the following two conditions are equivalent.
  \begin{enumerate}
  \item
    \label{enum:one-stratum-equiv-D+-1}
    The obvious triangulated functor
    is an equivalence
    \begin{equation}
      \D^+(\Loc(X)) \sira \D^+_\Loc(X).
    \end{equation}
  \item
    \label{enum:one-stratum-H>0-univ-covering-1}
    $\HH^{>0}(\tilde{X}, \ul{M})=0$ for all
    $R$-modules $M$, i.\,e.\ the universal covering of each
    path component of $X$ is acyclic.
  \end{enumerate}
\end{theorem}

\subsection{Constructible sheaves}

\begin{definition}
  \label{d:stratification-1}
  Let $X$ be a topological space.  A \define{stratification of
    $X$} is a finite partition $\mathcal{S}$ of $X$ into
  non-empty locally closed subsets, called \define{strata}, such
  that the closure of each stratum is a union of strata,
  i.\,e.\
  $\ol{S}=\bigcup_{T \subset \ol{S}} T$ for each
  $S \in \mathcal{S}$.  A \define{stratified space} is a
  topological space together with a stratificaton.
  The following conditions on a stratification
  $\mathcal{S}$ of $X$ are used.

  \begin{enumerate}
  \item[\mylabel{enum:loc-sc}{(loc-sc)}]
    Each stratum is locally simply connected.
  \item[\mylabel{enum:loc-sa}{(loc-sa)}]
    Each stratum is locally 1-acyclic.
  \end{enumerate}
\end{definition}

\begin{definition}
  Let $(X, \mathcal{S})$ be a stratified space. A
  sheaf $F$ on $X$ is \define{$\mathcal{S}$-constructible}
  if $F|_S \in \Loc(S)$ for all $S \in \mathcal{S}$.
  We write $\Cons(X, \mathcal{S})$ for
  the full subcategory of $\Sh(X)$ of
  $\mathcal{S}$-constructible objects. (So $\Cons(S,\{S\})=\Loc(S)$.)

  A complex $F$ in $\Sh(X)$ is
  \define{$\mathcal{S}$-constructible}
  if all its cohomology sheaves $\HH^p(F)$ are
  $\mathcal{S}$-constructible.
  We write
  $\D_\mathcal{S}(X):=\D_{\Cons(X,\mathcal{S})}(X)$
  for
  the full subcategory of $\D(X)$ of
  $\mathcal{S}$-constructible objects.
\end{definition}

\begin{proposition}
  \label{p:loc-sc-grothendieck-1}
  Let $(X, \mathcal{S})$ be a
  \ref{enum:loc-sc}-stratified space.
  Then $\Cons(X, \mathcal{S})$ is a Grothendieck abelian category
  and the
  inclusion functor
  $\Cons(X, \mathcal{S}) \ra \Sh(X)$ is exact and cocontinuous;
 in
  particular, coproducts in $\Sh(X)$ of families of
  objects in $\Cons(X, \mathcal{S})$ are again in $\Cons(X,
  \mathcal{S})$.

  Assume in addition that the stratification is
  \ref{enum:loc-sa}. Then
  $\Cons(X, \mathcal{S})$ is a weak Serre subcategory of
  $\Sh(X)$.  In particular, $\D_\mathcal{S}(X)$ is a thick
  triangulated subcategory of $\D(X)$
  which is closed under
  coproducts in $\D(X)$.
\end{proposition}

\subsection{Realization functor} Let $(X, \mathcal{S})$ be a
  \ref{enum:loc-sc}-\ref{enum:loc-sa}-stratified space. Then
  $\Cons(X, \mathcal{S})$ is an abelian category and $\D_\mathcal{S}(X)$ is a thick triangulated subcategory of $\D(X)$. One may ask if the obvious
  {\it realization} triangulated functor
  \begin{equation}
    \label{eq:real}
    \real \colon \D(\Cons(X, \mathcal{S})) \ra \D_{\mathcal{S}}(X).
  \end{equation}
is an equivalence.

\begin{definition} Let $(X, \mathcal{S})$ be a
  \ref{enum:loc-sc}-\ref{enum:loc-sa}-stratified space.
  This stratification $\mathcal{S}$ is
  called \ref{enum:l_*-preserves-cons} if it satisfies the
  following additional condition.
  \begin{enumerate}
  \item[\mylabel{enum:l_*-preserves-cons}{(cons)}]
    If $l \colon S \ra X$ is the inclusion of any stratum $S \in
    \mathcal{S}$, then
    \textcolor{red}{(cf. \cite[2.1.13, p.~61]{BBD})}
      the right derived functor $\Rd l_* \colon \D^+(S) \ra
      \D^+(X)$
      of $l_* \colon \Sh(S) \ra \Sh(X)$
      maps objects of $\Loc(S)$ (viewed as complexes concentrated
      in degree zero) to
      $\D^+_\mathcal{S}(X)$.
    \end{enumerate}
 \end{definition}

 \begin{lemma}\label{ext of funct and commute with real}
 Let $(X, \mathcal{S})$ be a
  \ref{enum:loc-sc}-\ref{enum:loc-sa}-\ref{enum:l_*-preserves-cons}
  stratified space and let $l \colon S \ra X$ be the inclusion of any stratum $S \in \mathcal{S}$.  Then the following holds.

  (1) The functor $l_* \colon \Sh(S) \ra \Sh(X)$ restricts to the left exact functor between the abelian subcategories
  $$l_*\colon \Loc (S)\ra \Cons(X, \mathcal{S})$$
  Denote by
  $$\Rd_\cs l_* \colon \D^+(\Loc (S)) \ra \D^+(\Cons(X, \mathcal{S}))$$
  its derived functor.

  (2)   Then there is a canonical 2-morphism
  \begin{equation}
    \label{eq:real-e_*-sigma-1}
    \sigma \colon \real \circ \Rd_\cs e_* \Ra \Rd e_* \circ \real
  \end{equation}
  between functors from $\D^+(\Loc(S))$ to
  $\D^+(X)$, as illustrated by the diagram
  \begin{equation*}
    \xymatrix{
      {\D^+(\Loc(S))}
      \ar[r]^-{\Rd_\cs l_*} \ar[d]_-{\real} &
      {\D^+(\Cons(X, \mathcal{S}))}
      \ar[d]^-{\real}
      \ar@{=>}[ld]_-{\sigma}\\
      {\D^+(S)}
      \ar[r]^-{\Rd l_*} &
      {\D^+(X).}
    }
  \end{equation*}
 \end{lemma}

 \begin{theorem}
  \label{t:equivalent-conditions-for-equivalence-1}
  For a
  \ref{enum:loc-sc}-\ref{enum:loc-sa}-\ref{enum:l_*-preserves-cons}-stratified
  space $(X, \mathcal{S})$, the following conditions are equivalent.
  \begin{enumerate}
  \item
    \label{enum:equivalence-X}
    The functor
    \begin{equation}
      \label{eq:equi-D+-X}
      \real \colon \D^+(\Cons(X, \mathcal{S})) \ra
      \D_\mathcal{S}^+(X)
    \end{equation}
    is an equivalence.
  \item
    \label{enum:equiv-S-and-sigma-iso}
    For all strata $S \in \mathcal{S}$,
    the functor
    \begin{equation}
      \label{eq:equi-D+-S-1}
      \real \colon \D^+(\Loc(S)) \ra
      \D_\Loc^+(S)
    \end{equation}
    is an equivalence (cf.\
    Remark~\ref{r:stratawise-equiv-1}),
    and
    \begin{equation}
      \label{eq:equi-D+-sigma-1}
      \sigma \colon \real \circ \Rd_\cs l_*
      \xRa{\eqref{eq:real-e_*-sigma-1}} \Rd l_* \circ
      \real
    \end{equation}
    is a 2-isomorphism, where $l
    \colon S \ra X$ denotes the inclusion.
    \end{enumerate}
    \end{theorem}

\begin{remark} \label{r:stratawise-equiv-1} By theorem \ref{t:main-one-stratum-1} the functor \eqref{eq:equi-D+-S-1} in the above theorem is an equivalence if $S$ is locally simply connected, locally acyclic and $\HH^{>0}(\tilde{S}, \ul{M})=0$ for all
    $R$-modules $M$, where $\tilde{S}$ is the disjoint union of the universal coverings of path component of $S$.
\end{remark}

Next we  give some sufficient conditions for the 2-morphism \eqref{eq:equi-D+-sigma-1} to be an isomorphism.

For simplicity we restrict ourselves to stratified spaces $(X,\mathcal{S})$ with a {\it normal structure} (see \ref{def-normal-str} for precise definition). This means that all strata are connected manifolds and for every stratum $T$ and every point $x\in T$ there exists an open neighbourhood $x\in U\subset X$ which is homeomorphic to the product of $U\cap T$ with the cone over the {\it link} $L_x$ of $T$ at $x$. If $S$ is another stratum s.t. $T\subset \overline{S}$, then $L_{x,S}:=L_x \cap S$ is a manifold with finitely many connected components $L_{x,S,i}$. The next result follows from Theorem
\ref{our-original-example}.

\begin{theorem}\label{our-original-example-1} Let $(X,\mathcal{S})$ be a stratified space with a normal structure. Assume that each stratum $S$ is a $K(\pi ,1)$ manifold, i.e. its universal covering space is contractible. Assume in addition that for any strata $S,T$, such that $T\subset \overline{S}$,  and any point $x\in T$, each connected component $L_{x,S,i}$ of the manifold $L_{x,S}$ is also a $K(\pi ,1)$-space.

Then the equivalent conditions of Theorem  \ref{t:equivalent-conditions-for-equivalence-1} hold if for each pair of strata $T\subset \overline{S}$, each point $x\in T$ and each connected component
$L_{x,S,i}$ of the manifold $L_{x,S}$ the following holds:
The kernel of the homomorphism $\pi _1(L_{x,S,i})\to \pi _1(S)$ is a finite subgroup whose order is prime to the characteristic of $R$. (In particular, if this homomorphism is injective)
\end{theorem}

For example, the functor
$$ \real \colon \D^+(\Cons(X, \mathcal{S})) \ra
      \D_\mathcal{S}^+(X)$$
is an equivalence for a complex toric variety $X$ with stratification by torus orbits (Corollary \ref{toric-var}) and it is not an equivalence for $X=\bP ^1\bC $ with the Bruhat stratification $X=\bC \sqcup \{pt\}$.

\subsection{Local systems of finite rank} In Section \ref{sec-versions} we consider a related question for constructible sheaves of {\it finite type}. Namely, assume that the coefficient ring $R$ is a field $k$ and denote by
$\Cons _{ft}(X,\mathcal{S})\subset \Cons(X,\mathcal{S})$ the full abelian subcategory of constructible sheaves with finite dimensional stalks. Let also $\D ^b_{ft}(X)\subset \D ^b_{\mathcal{S}}(X)$ be the corresponding triangulated category. One may ask when the obvious functor
\begin{equation}\label{a version-1}
\D ^b(\Cons _{ft}(X,\mathcal{S}))\to \D^b _{ft}(X)
\end{equation}
is an equivalence. In particular, is \eqref{a version-1} an equivalence when \eqref{eq:equi-D+-S-1} is such? The answer would be positive if the obvious functor
\begin{equation}\label{a version-2}
\D ^b(\Cons _{ft}(X,\mathcal{S}))\to \D ^b_{\Cons _{ft}}(\Cons (X,\mathcal{S}))
\end{equation}
is also an equivalence. We study questions of this sort in Section
\ref{sec-versions} obtaining positive answers in some cases (see for example Corollary \ref{cor vers toric var}).

Actually it is even not clear when \eqref{a version-2} is an equivalence in the case of one stratum. This is a purely algebraic question that we resolve in some particular cases in Section \ref{sec-versions}.

\section{Locally constant sheaves: Proof of Proposition \ref{p:fiber-functor-1}}
\label{sec:locally-const-sheav}

\begin{para}
  \label{para:lc-espace-etale-covering-map-1}
  Let $F$ be a sheaf on a topological space
  $X$ and $\tilde{F}$ its espace \'etal\'e. Then $F$ is locally
  constant if and only if
  $\tilde{F} \ra X$ is a covering map. \textcolor{red}{}
\end{para}

\begin{para}
  Recall that $\Sh(X)$ is a Grothendieck abelian category
  \cite[\sptag{079V}, \sptag{01AH}]{stacks-project}.
\end{para}

\begin{proposition}
  \label{p:Loc-simply-connected-1}
  Let $X$ be a simply connected,
  locally path connected
  topological space. Then the following statements are true.
  \begin{enumerate}
  \item
    \label{enum:equiv-products}
    The functor $\Mod(R) \ra \Sh(X)$,
    $M \mapsto \ul{M}$ is fully faithful and exact with essential
    image
    $\Loc(X)$ and hence provides an equivalence
    \begin{equation*}
      \Mod(R) \sira \Loc(X)
    \end{equation*}
      of categories.
      Taking global sections is a
      quasi-inverse of
      this functor. (Another quasi-inverse is taking the stalk at an
      arbitrary point.)
    In particular, $\Loc(X)$ has all categorical properties of
    $\Mod(R)$, e.\,g.\ it is a (complete and cocomplete)
    Grothendieck abelian
    category with enough injectives and projectives.

     Moreover,
    the inclusion functor $\Loc(X) \ra \Sh(X)$ is exact,
    continuous (= preserves limits) and cocontinuous (= preserves
    colimits).

  \item
    \label{enum:weak-Serre}
    $\Loc(X)$ is closed under extensions in $\Sh(X)$ (i.e. $\Loc(X)$
    is a weak Serre subcategory
    of $\Sh(X)$) if and only if $X$ is 1-acyclic.
 \end{enumerate}
\end{proposition}

\begin{proof}
  \ref{enum:equiv-products}
  Let $c \colon X \ra \pt$ be the map to a one-point-space,
  viewed as a map of ringed spaces with structure sheaves
  $\ul{R}$ and $R$, respectively.
  Then we have an adjunction $c^* \colon \Mod(R)
  \leftrightarrows \Sh(X) \colon c_*$.
  Note that $c^*(M)=\ul{M}$ and $c_*(F)=\Gamma(X,F)$.
  We need to see that $c^*$
  is fully faithful and that its essential image consists
  precisely of all locally constant sheaves.

  Given $F \in \Loc(X)$ consider its espace \'etal\'e
  $\tilde{F}$ which is a module object
  over the ring object $\ul{\tilde{R}}$ in the category
  $\Top_{/X}$ of topological spaces over $X$, cf.\
  \cite[II.7]{maclane-moerdijk}.  Since $F$ is locally constant,
  any $x \in X$ has an open neighborhood such that $F|_U$ is
  constant, i.\,e.\ $\tilde{F}|_U
  \cong N \times U$
  (as module
  objects) compatible with the projections to $U$, where $N$ is
  an $R$-module
  equipped with the discrete topology. This means that
  $\tilde{F}$ is locally path connected and $\tilde{F} \ra X$ is
  a covering map.
  Hence each
  connected component (= path
  component) of $\tilde{F}$ is open in $\tilde{F}$ and - $X$
  being simply connected
  and locally path connected -
  maps homeomorphically onto $X$.
  If $M$ is the set of path component of $\tilde{F}$, we obtain a
  homeomorphism $\tilde{F} \sira M \times X$ over $X$ by mapping
  a point to the pair consisting of its image in $X$ and its path
  component. Taking
  the action of $\ul{\tilde{R}}$
  into account
  shows that $M$ has
  a natural structure of $R$-module such that this
  is in fact a homeomorphism of module objects over
  $\ul{\tilde{R}}$ in
  $\Top_{/X}$. Hence $\ul{M} \cong F$ as sheaves. This shows
  essential surjectivity.

  Fully faithfulness of $c^*$ is equivalent to $M \ra
  c_*c^*(M)=\Gamma(X,\ul{M})$ being an isomorphism
  for each $R$-module $M$. But this is clear since any (global)
  section of $\ul{\tilde{M}} = M \times X \ra X$ has the form $x
  \mapsto (x,m)$ for a fixed $m \in M$ since $X$ is connected and
  $M$ has the discrete topology.
  This proves the equivalence $\Mod(R) \sira \Loc(X)$.

  The fact that $c^* \colon \Mod(R) \ra \Sh(X)$ is exact and
  left adjoint then
  implies that
  $\Loc(X) \ra \Sh(X)$ is exact and
  cocontinuous (= preserves colimits)

  In order to show that the inclusion functor preserves limits it
  is enough to show that
  $c^*$ preserves products. Let
  $(M_i)_{i \in I}$ be a family of $R$-modules. Since the
  presheaf product
  of sheaves is already a sheaf we obtain for an open subset $U$
  of $X$
  \begin{equation*}
    (\prod_{i \in I} \ul{M_i}) (U)
    = \prod_{i \in I} (\ul{M_i}(U))
    = \prod_{i \in I} (\Top_{/U}(U, M_i \times U))
    = \prod_{i \in I} (\Top(U, M_i))
    = \Top(U, \prod_{i \in I} M_i)
  \end{equation*}
  where $T:=\prod_{i \in I} M_i$ has the product topology (= Tychonoff
  topology).
  On the other hand
  \begin{equation*}
    \ul{\left(\prod_{i \in I} M_i\right)}(U)
    = (\Top_{/U}(U, (\prod_{i \in I} M_i) \times U))
    = \Top(U, \prod_{i \in I} M_i)
  \end{equation*}
  where $D:=\prod_{i \in I} M_i$ has the discrete
  topology.
  The identity is a continous map $D \ra T$. We claim that the
  induced map
  \begin{equation*}
    \Top(U,D) \ra \Top(U,T)
  \end{equation*}
  is an isomorphism. Injectivity is trivial.
  For surjectivity
  let $s \colon U \ra T$ be a continuous map. Since $U$ is
  locally path connected, $U$ is the disjoint union of its open
  path components, i.\,e.\ the coproduct in $\Top$ of these path
  components. Hence we can assume without loss of generality that
  $U$ is (path) connected. Then each composition $U
  \xra{s} T
  \ra M_i$ must be constant since $M_i$ has the discrete
  topology. Hence $s$ is constant and
  continuous as a map $U \ra D$. This shows surjectivity.

  \ref{enum:weak-Serre} Assume that $\HH^1(X,
    \ul{K})=0$ for all $R$-modules  $K$.
  Let
  $0 \ra M \ra E \ra N \ra 0$ be a short exact sequence
  in $\Sh(X)$ with $M,N \in \Loc(X)$.
  By part~\ref{enum:equiv-products} we have
  $\ul{\Gamma(X,M)} \sira M$ and therefore
  $\HH^1(X,M)=0$ by assumption.
  The long exact sequence for the derived functors of
  $\Gamma(X,-)$ therefore shows that
  $0 \ra \Gamma(X,M) \ra \Gamma(X,E) \ra \Gamma(X,N) \ra 0$ is
  exact.
  Apply the exact functor $c^*$ to this short exact sequence and
  consider the
  morphism to our original short exact sequence given by
  adjunction counits.
  Since the adjunction counits are isomorphisms for the locally
  constant sheaves $M$ and $N$
  by part~\ref{enum:equiv-products}, $\ul{\Gamma(X,E)} \ra E$ is
  an isomorphism as well, i.\,e.\ $E \in \Loc(X)$.

  Vice versa assume that $\Loc(X)$ is closed under extensions in
    $\Sh(X)$.
    Let $\xi \in \HH^1(X,
    \ul{M})=\Ext_{\Sh(X)}^1(\ul{R}, \ul{M})$
    be an element. Let
    $0 \ra \ul{M} \ra E \ra \ul{R} \ra 0$
    be a short exact sequence
    in $\Sh(X)$
    whose Yoneda equivalence class corresponds
    to $\xi$.
    Since $\Loc(X)$ is closed under extensions by assumption, our
    short exact sequence lives in $\Loc(X)$.
    The equivalence $\Mod(R) \cong \Loc(X)$ from part
    \ref{enum:equiv-products}
    shows that $\ul{R}$ is projective in $\Loc(X)$. Hence our
    short exact sequence splits.
    This implies $\xi=0$.

\end{proof}

\begin{proposition}
  \label{p:Loc-lsc-lsa-1}
  Let $X$ be a locally simply connected
  topological space.
  Then $\Loc(X)$ is a Grothendieck abelian category and the
  inclusion functor
  $\Loc(X) \ra \Sh(X)$ is exact, continuous and cocontinuous; in
  particular, coproducts and products in $\Sh(X)$ of families of
  objects in $\Loc(X)$ are again in $\Loc(X)$.

If in addition $X$ is locally 1-acyclic, then
  $\Loc(X)$ is a weak Serre subcategory of $\Sh(X)$.
\end{proposition}

\begin{proof}
  Let $X$ be open-locally simply connected. Let $K=\kernel(f)$ be the
  kernel
  in $\Sh(X)$ of a morphism $f \colon M \ra N$ in $\Loc(X)$.
  In order to show that $K$ is locally constant let $x \in X$.
  Let $U$ be a simply connected open neighborhood of $x$ in $U$.
  Note that $j^* \colon \Sh(X) \ra \Sh(U)$ is exact
  and maps $\Loc(X)$ to $\Loc(U)$. Hence $j^*(K) =
  \kernel(j^*(f))) \in \Loc(U)$ by
  part~\ref{enum:equiv-products}
  of Proposition~\ref{p:Loc-simply-connected-1} since $U$ is simply
  connected and locally simply connected, hence locally
  path
  connected. Hence $K \in \Loc(X)$.
  Similarly one proves that the cokernel in $\Sh(X)$ of any
  morphism in $\Loc(X)$ is in $\Loc(X)$. This shows that
  $\Loc(X)$ is an abelian subcategory of $\Sh(X)$.

  The functor $j^*=j^! \colon \Sh(X) \ra \Sh(U)$ has a left adjoint
  $j_!$ and a right adjoint $j_*$ and hence preserves products
  and coproducts.
  Again from
  part~\ref{enum:equiv-products}
  of Proposition~\ref{p:Loc-simply-connected-1}
  we deduce that $\Loc(X)$ is closed
  under products and coproducts in $\Sh(X)$, and that all limits and
  colimits in
  $\Loc(X)$ exist and coincide with those computed in $\Sh(X)$.
  Since $\Sh(X)$ is a Grothendieck abelian category, so is
  $\Loc(X)$.

  In order to check that $\Loc(X)$ is closed under extensions we
  can again check this locally on a simply connected open subset
  $U$ as above. We may assume that $\HH^1(U, \ul{M})=0$ for all $R$-modules $M$
  so  part~\ref{enum:weak-Serre}
  of Proposition~\ref{p:Loc-simply-connected-1} applies.
\end{proof}

\subsection{$G$-spaces and equivariant sheaves}
  Let $G$ be a group and $X$ a topological space with a
  $G$-action by homeomorphisms, i.\,e.\ the map
  $g \colon X \ra X$ is continuous for all $g \in G$ or,
  equivalently, the action map $G \times X \ra X$ is continuous
  where $G$ carries the discrete topology.  Then there is the
  notion of a $G$-equivariant sheaf (of sets or $R$-modules) on
  $X$ (see \cite[5.1]{groth-tohoku}). We usually think
  of a $G$-equivariant sheaf $F$ (of sets or $R$-modules) on $X$
  in terms of its espace
  \'etal\'e, i.\,e.\ there is a continuous action
  $G \times \tilde{F} \ra \tilde{F}$ which is compatible with the
  given action $G \times X \ra X$ under the obvious maps.  Let
  $\Sh_G(X)$ denote the category of $G$-equivariant sheaves (of
  $R$-modules) on $X$.  This is a Grothendieck abelian category
  (see \cite[Prop.~5.1.1]{groth-tohoku}).  The forgetful functor
  $\Sh_G(X) \ra \Sh(X)$ is exact, continuous and cocontinuous.
  We call a $G$-equivariant sheaf
  \define{locally constant} if its underlying sheaf is locally
  constant.
  The full subcategory of $\Sh_G(X)$ of $G$-equivariant locally
  constant sheaves is denoted by $\Loc_G(X)$.

  Let $RG$ be the group ring of $G$ over $R$. Given $M \in
  \Mod(RG)$,
  let $\ul{M}$ be  the sheaf of local
  sections of the morphism
  $M \times X \ra X$ of topological $G$-spaces, where $M$ carries
  the discrete topology and $G$ acts diagonally on $M \times X$.
  Then clearly $\ul{M} \in \Loc_G(X)$.
  The underlying sheaf of this $G$-equivariant sheaf is the
  constant sheaf associated to the underlying $R$-module of $M$.

\begin{proposition}
  Let $X$ be a simply connected,
  locally path connected
  topological space.
  Let $G$ be a group acting on $X$ by homeomorphisms.
  Then the functor $\Mod(RG) \ra \Sh_G(X)$,
  $M \mapsto \ul{M}$, is fully faithful with essential image
  $\Loc_G(X)$ and hence provides an equivalence
  \begin{equation}
    \label{eq:ModRg-LocGX-1}
    \Mod(RG) \sira \Loc_G(X)
  \end{equation}
  of categories. (Taking global sections is a
    quasi-inverse of
    this functor.)
  In particular, $\Loc_G(X)$
  is a (complete and cocomplete)
  Grothendieck abelian category
  with enough injectives and projectives.
  Moreover,
  the inclusion functor $\Loc_G(X) \ra \Sh_G(X)$ is exact,
  continuous and cocontinuous.
  If $X$ is 1-acyclic, then
  $\Loc_G(X)$ is closed under extensions in $\Sh_G(X)$.
\end{proposition}

\begin{proof}
  This follows from (the proof of)
  Proposition~\ref{p:Loc-simply-connected-1} as follows.
  Given $F \in \Loc_G(X)$, we know that the underlying sheaf $F
  \in \Loc(X)$ has espace \'etal\'e $\tilde{F}$ canonically
  isomorphic to
  $M \times X$ where $M$ is the
  $R$-module of path components of 
  $\tilde{F}$.
  Since $F$ is $G$-equivariant, the action of $G$ permutes the
  path componentes of $\tilde{F}$, i.\,e.\ $G$ acts on $M$. This
  action is $R$-linear, so $M$ is an $RG$-module.
  Then clearly $\ul{M} \sira F$ as $G$-equivariant sheaves.
  Hence our functor is essentially surjective.
  Given $M,N \in \Mod(RG)$ we have
  $\Hom_{RG}(M,N)=\Hom_R(M,N)^G$ and
  $\Hom_{\Sh_G(X)}(\ul{M},\ul{N})=
  \Hom_{\Sh(X)}(\ul{M},\ul{N})^G$
  for the obvious $G$-action on $\Hom_{\Sh(X)}(\ul{M},\ul{N})$.
  Therefore fully faithfullness follows from fully faithfulness
  of $\Mod(R) \ra \Sh(X)$ by taking $G$-invariants.

  Part~\ref{enum:weak-Serre} of
  Proposition~\ref{p:Loc-simply-connected-1} implies the claim
  about closedness under extensions.
\end{proof}

\begin{definition}  Let $G$ be a group and $X$ a $G$-space. We say that the $G$-action is {\bf topologically free} (see
  \cite[\S 81]{munkres-topology-2nd}) if every point $x$ in $X$ has a neighbourhood $U$, such that $U$ and $gU$ are disjoint for every nonidentity $g\in G$.
\end{definition}

\begin{proposition}
  \label{p:LocY=LocGX-1}
  Let $G$ be a group and $X$ a $G$-space. Assume that the action
  of $G$ is topologically free. Let
  $\pi \colon X \ra Y:=G\backslash X$ be the quotient map.
  Then the pullback $\pi^*F$ of any sheaf $F \in \Sh(Y)$ has a
  natural structure of $G$-equivariant sheaf, and this defines
  equivalences $\pi^* \colon \Sh(Y) \sira \Sh_G(X)$ and
  \begin{equation}
    \label{eq:LocY-LocGX-1}
    \pi^* \colon \Loc(Y) \sira \Loc_G(X).
  \end{equation}
  A (quasi-)inverse of these equivalences is given by mapping a
  $G$-equivariant sheaf $F$ on $X$ to the subsheaf $\pi_*^G(F)$
  of $G$-invariant sections of its pushforward $\pi_*F$. The
  espace \'etal\'e of $\pi_*^G(F)$ is the quotient
  space $G \backslash \tilde{F}$ of the espace \'etal\'e
  $\tilde{F}$ of $F$.
\end{proposition}

\begin{proof}
  The fact that $\pi^* \colon \Sh(Y) \sira \Sh_G(X)$ is an
  equivalence is well-known, see e.\,g.\
  \cite[5.1, page~199]{groth-tohoku}, and easy to prove (together
  with the description of the inverse functor).
  Then \eqref{eq:LocY-LocGX-1} follows because a sheaf on $Y$ is
  locally constant if and only if its pullback to $X$ is locally
  constant; this uses that $\pi$ is a covering map.
\end{proof}

\begin{proposition}
  \label{p:proj-generator-1}
  Let $G$ be a group and $X$ a $G$-space with a topologically free $G$-action and quotient map
  $\pi \colon X \ra G \backslash X=:Y$.  Assume that $X$ is
  simply connected and locally path connected.
  Then the direct image
  with proper support $\pi_!\ul{R}$ of the
  (non-equivariant) constant sheaf $\ul{R} \in \Loc(X)$ is
  locally constant and a
  projective generator of $\Loc(Y)$. More precisely,
  it
  corresponds to $RG \in \Mod(RG)$ under the equivalence
  \begin{equation}
    \label{eq:ModRG-LocY-1}
    \Mod(RG) \sira \Loc(Y)
  \end{equation}
  obtained from the equivalence \eqref{eq:ModRg-LocGX-1} and the
  quasi-inverse $\pi_*^G$ of the equivalence
  \eqref{eq:LocY-LocGX-1}.
\end{proposition}

\begin{proof}
  The espace \'etal\'e of the $G$-equivariant locally constant
  sheaf $\ul{RG}$ is $RG \times X$. The map $RG \ra R$,
  $\sum r_g g \mapsto r_e$, induces a morphism
  $RG \times X \ra R \times X$, or equivalently, a morphism
  $\epsilon \colon \ul{RG} \ra \ul{R}$ in $\Sh(X)$. We claim
  more precisely
  that the
  morphism $\pi_*\epsilon \colon \pi_*\ul{RG} \ra \pi_*\ul{R}$
  restricts to an
  isomorphism $\pi^G_*\ul{RG} \sira
  \pi_!\ul{R}$ on the specified subsheaves.

  Since our action is topologically free, $Y$ is locally
  path
  connected. Moreover,
  any open neighborhood of any point of $Y$ contains an open
  connected
  neighborhood of that point such that there is an isomorphism
  $\pi^{-1}(V) \sira G \times V$ which is $G$-equivariant with
  respect to the $G$-action
  $g.(h,v)=(gh,v)$ on $G \times V$.
  Hence it is enough to show that
  $(\pi_*\epsilon)(V) \colon (\pi_*\ul{RG})(V) \ra
  (\pi_*\ul{R})(V)$ identifies
  $(\pi^G_*\ul{RG})(V)$ with
  $(\pi_!\ul{R})(V)$ for any such $V$. Fix $V$ as above.
  Then
  \begin{align*}
    (\pi_*^G\ul{RG}) (V)
    & = \{\text{$s \colon \pi^{-1}(V) \ra RG$ continuous
      $G$-equivariant}\}\\
    & \cong \{\text{$s \colon G \times V \ra RG$ continuous
      $G$-equivariant}\}\\
    \text{($V$ connected non-empty)} \quad
    & \cong \{\text{$s \colon G \ra RG$ $G$-equivariant}\}\\
    \text{(evaluation at $e \in G$)} \quad
    & \sira RG
  \end{align*}
  where an element $\sum_{g \in G} r_g g \in RG$ corresponds to
  the continuous equivariant
  section $G \times V \ra RG$, $(h,v) \mapsto \sum_{g \in G}
  r_g hg$. Under $(\pi_*\epsilon)(V)$, this section is mapped to the
  continuous section $(h,v) \mapsto r_{h^{-1}}$ whose support
  consists of finitely many sheets
  (namely those sheets $h \times V$ with $r_{h^{-1}}\not=0$)
  of $G \times V \ra V$.
  On the other hand, we have\footnote{Same for $\pi_!$ with
    additional requirement ``support proper''.}
  \begin{align*}
    \pi_!(\ul{R}) (V)
    & = \{\text{$s \colon \pi^{-1}(V) \ra R$ continuous with
      $\supp(s) \ra V$ proper}\}\\
    & \cong \{\text{$s \colon G \times V \ra R$ continuous with
      $\supp(s) \ra V$ proper}\}\\
    \text{($V$ connected non-empty)} \quad
    & \cong \{\text{$s \colon G \ra R$ with
      $\supp(s)$ finite}\}\\
    & = RG
  \end{align*}
  where an element
  $\sum_{g \in G} t_g g \in RG$ corresponds to
  the continuous
  section $G \times V \ra R$, $(h,v) \mapsto t_h$, with support
  on
  finitely many sheets.
  Hence $\pi_*\epsilon(V)$ maps
  $(\pi_*^G\ul{RG}) (V)$ bijectively to
  $\pi_!(\ul{R}) (V)$.
  (The corresponding endomorphism of $RG$ maps
  $\sum r_g g \in RG \cong (\pi_*^G\ul{RG}) (V)$ to
  $\sum r_{g^{-1}}g \in RG \cong \pi_!(\ul{R}) (V)$).
\end{proof}

\begin{definition}
  A topological space $X$ is
  \define{semilocally simply connected} if any $x \in X$
  has a neighborhood (which may be assumed to be open)
  $U$ such that $\pi_1(U,x) \ra
  \pi_1(X,x)$ is trivial.
\end{definition}

\begin{para}
  \label{para:existence-universal-covering-1}
  Recall that a non-empty, path connected, locally path connected
  topological space $X$
  has a universal covering if and only if it is semilocally
  simply connected (see \cite[Cor.~82.2]{munkres-topology-2nd}).

  Clearly, any locally simply connected space is locally path
  connected and semilocally simply connected, and any open
  subspace is again locally simply connected.
  Hence any non-empty, path connected,
  locally simply connected topological space
  and each of its non-empty path connected open
  subspaces
  has a universal covering.

  Any two universal coverings of a non-empty path connected
  topological space are isomorphic; they are uniquely isomorphic
  if we work with pointed spaces. Therefore we speak about the
  universal covering even though this is a bit sloppy.
\end{para}

\begin{proposition}
  \label{p:fiber-functor-2}
  Let $X$ be non-empty, path connected, locally path connected
  topological space. Let $x \in X$. Then mapping $L \in \Loc(X)$
  to its fiber $L_x$ with its $R$-module structure and left
  action of $\pi_1(X,x)$ by lifting loops at $x$ to paths in the
  espace \'etal\'e $\tilde{L}$ (cf.\
  \ref{para:lc-espace-etale-covering-map-1})
  defines an equivalence ``take fiber at $x$''
  \begin{equation*}
    \Loc(X) \sira \Mod(R \pi_1(X,x)).
  \end{equation*}
\end{proposition}

\begin{proof}
  We know by~\ref{para:existence-universal-covering-1} that
  $X$ has a universal covering $\tilde{X} \ra X$. The group $G$
  of covering transformations of $\tilde{X}$ acts properly
  discontinuously on $\tilde{X}$, and there is a canonical
  identification $G = \pi_1(X,x)$.
  Hence $\Mod(\pi_1(X,x))=\Mod(RG)$, and it is easy to see that
  the equivalence $\Mod(RG) \sira \Loc(X)$
  from \eqref{eq:ModRG-LocY-1}
  in
  Proposition~\ref{p:proj-generator-1}
  is quasi-inverse to the functor ``take fiber at $x$''.
\end{proof}

\subsection{Proof of Proposition \ref{p:fiber-functor-1}} Proposition \ref{p:fiber-functor-1} follows immediately from Proposition \ref{p:fiber-functor-2} and Proposition \ref{p:Loc-lsc-lsa-1}.

\section{Derived category $\D(\Loc(X))$: proof of Theorem \ref{t:main-one-stratum-1}}

Let $X$ be a locally simply connected topological space. Then
$\Loc(X)$ is abelian
by
Proposition~\ref{p:Loc-lsc-lsa-1}, so $\D(\Loc(X))$ is defined.
Moreover, $\Loc(X) \ra \Sh(X)$ being exact, there is an obvious functor $\D(\Loc(X)) \ra \D(X)$. It commutes with
arbitrary coproducts and lands in
$\D_\Loc(X)$.

If in addition $X$ is locally 1-acyclic, then $\Loc(X)$ is a weak Serre
subcategory of $\Sh(X)$ (see
Proposition~\ref{p:Loc-lsc-lsa-1}) then $\D_\Loc(X)$ is a thick triangulated
subcategory of $\D(X)$.
  It is closed
  under arbitrary coproducts in
  $\D(X)$.
  The obvious functor
  \begin{equation*}
    \D(\Loc(X)) \ra \D_\Loc(X)
  \end{equation*}
  is triangulated. All categories $\D(X)$, $\D(\Loc(X))$
  and $\D_\Loc(X)$ have all coproducts and are in particularly
  idempotent complete. The above functor commutes with coproducts.

\begin{theorem}
  \label{t:main-one-stratum-2} (=Theorem \ref{t:main-one-stratum-1})
  Let $X$ be a
  locally simply connected,
  locally acyclic
  topological space. Let $\tilde{X}$ be the disjoint union of the
  universal coverings of all path components of $X$.
  Then the following two conditions are equivalent.
  \begin{enumerate}
  \item
    \label{enum:one-stratum-equiv-D+-2}
    The obvious triangulated functor
    is an equivalence
    \begin{equation}
      \label{eq:one-stratum-equiv-2}
      \D^+(\Loc(X)) \sira \D^+_\Loc(X).
    \end{equation}
  \item
    \label{enum:one-stratum-H>0-univ-covering-2}
    $\HH^{>0}(\tilde{X}, \ul{M})=0$ for all
    $R$-modules $M$, i.\,e.\ the universal covering of each
    path component of $X$ is acyclic.
  \end{enumerate}

\end{theorem}

\begin{proof}
  Let $\pi \colon \tilde{X} \ra X$ be the obvious map.
  We give two more conditions.
  \begin{enumerate}[label=(\alph*)']
  \item
    \label{enum:one-stratum-equiv-Db}
    The obvious triangulated functor
    is an equivalence
    \begin{equation*}
      \D^\bd(\Loc(X)) \sira \D^\bd_\Loc(X).
    \end{equation*}
  \item
    \label{enum:proj-gen-to-Loc}
    The map
    \begin{equation}
      \label{eq:proj-gen-to-Loc-1}
      \Hom_{\D(\Loc(X))}(\pi_! \ul{R}, [q]N) \ra
      \Hom_{\D(X)}(\pi_! \ul{R}, [q]N)
    \end{equation}
    is an isomorphism for all $N \in \Loc(X)$ and $q \in \bZ$.

  \end{enumerate}
  We will show that the given four conditions are equivalent.
  The implications
  \ref{enum:one-stratum-equiv-D+-2} $\Rightarrow$
  \ref{enum:one-stratum-equiv-Db} and
  \ref{enum:one-stratum-equiv-Db} $\Rightarrow$
  \ref{enum:proj-gen-to-Loc} are clear.
  We can assume without loss of generality that $X$ is
  connected, so $\pi \colon \tilde{X} \ra X$ is the universal
  covering. Let $G$ be its group of covering
  transformations. Then $G$ acts topologically freely on
  $\tilde{X}$ and $G \backslash \tilde{X}=X$ canonically.

  \ref{enum:proj-gen-to-Loc}
  $\Longleftrightarrow$
  \ref{enum:one-stratum-H>0-univ-covering-2}:
  Clearly, the morphism \eqref{eq:proj-gen-to-Loc-1} is an
  isomorphism for
  $q=0$, and its source
  vanishes for all $q\not=0$ because
  $\pi_!\ul{R}$ is projective in $\Loc(X)$, by
  Proposition~\ref{p:proj-generator-1}.

  Note that there is an adjunction
  $(\pi_!, \pi^*)$ of exact functors between $\Sh(\tilde{X})$
  and $\Sh(X)$
  Hence the
  target of the morphism
  \eqref{eq:proj-gen-to-Loc-1} is isomorphic to
  \begin{equation*}
    \Hom_{\D(\tilde{X})}(\ul{R}, [q]\pi^*N) \cong
    \HH^q(\tilde{X}, \pi^*N).
  \end{equation*}
   Hence it suffices to show that
  the set of isoclasses of
  (non-equivariant) locally constant sheaves
  $\pi^*N$, for $N \in \Loc(X)$, coincides with the set of
  isoclasses of constant sheaves $\ul{M}$, for $M \in \Mod(R)$.
  But this is clear by the equivalence
  $\pi^* \colon \Loc(Y) \sira \Loc_G(X)$
  (Proposition~\ref{p:LocY=LocGX-1}), essential surjectivity
  of $\Loc_G(X) \ra \Loc(X)$ (use the trivial $G$-action), and the
  equivalence $\Mod(R) \sira \Loc(X)$
  (Proposition~\ref{p:Loc-simply-connected-1}.\ref{enum:equiv-products}).

  \ref{enum:proj-gen-to-Loc}  $\Rightarrow$
  \ref{enum:one-stratum-equiv-D+-2} :
  Any projective object of
  $\Loc(X)$ is a summand of a coproduct of copies of the
  projective generator
  $\pi_!\ul{R}$.
  Therefore, the map
  \begin{equation}
    \label{eq:proj-to-Loc}
    \Hom_{\D(\Loc(X))}(P, [p]N) \ra
    \Hom_{\D(X)}(P, [p]N)
  \end{equation}
  is bijective for all projective objects $P \in \Loc(X)$, all
  objects $N \in \Loc(X)$ and all $p \in \bZ$. Using d\'evissage (and the intelligent truncation) we
  deduce by standard argument that \begin{equation}
    \label{eq:D-D+-1}
    \Hom_{\D(\Loc(X))}(M, N) \ra
    \Hom_{\D(X)}(M, N)
  \end{equation}
  is an isomorphism for all $M \in \D^-(\Loc(X))$ and all
  $N \in \D^+(\Loc(X))$.

  We now show that \eqref{eq:D-D+-1}
  is even
  an isomorphism for all $M \in \D(\Loc(X))$ and all
  $N \in \D^+(\Loc(X))$.

  Since (countable) coproducts in $\Loc(X) \cong
  \Mod(RG)$ exist and are
  exact, any object $M \in \D(\Loc(X))$ can be presented as the
  homotopy
  colimit of its truncations $(\tau_{\leq n}M)_{n \in \bN}$,
  i.\,e.\ there is a triangle
  \begin{equation*}
    \bigoplus_{n \in \bN} \tau_{\leq n}M
    \xra{1-\text{shift}}
    \bigoplus_{n \in \bN} \tau_{\leq n}M
    \ra M
    \ra
  \end{equation*}
  in $\D(\Loc(X))$,
  by \cite[Lemma~2.15, Prop.~2.16]{bergh-schnuerer-gerbes}
  (cf.\ \cite[Rem.~2.3]{neeman-homotopy-limits}).
  The image of this triangle in $\D(X)$ is certainly a triangle;
  moreover, it exhibits $M$ as the
  homotopy
  colimit of its truncations $(\tau_{\leq n}M)_{n \in \bN}$
  since $\D(\Loc(X)) \ra \D(X)$ preserves coproducts.
  Now applying the cohomological functors $\Hom_{\D(\Loc(X))}(-,N)$
  and $\Hom_{\D(X)}(-,N)$ to these two triangles, using the
  universal property of the coproduct, boundedness from
  above of the objects $\tau_{\leq n}M$ and our previous
  knowledge proves the claim.

  In particular, this shows that
  $\D^+(\Loc(X)) \ra \D^+_\Loc(X)$ is full and faithful.
  It remains to show essential surjectivity.
  Clearly, all objects of $\Loc(X)$ and their shifts are in the
  essential image. Then, by intelligent truncation and fullness,
  we see that $\D^\bd_\Loc(X)$ is contained in the essential
  image.
  Now let $F \in \D^+_\Loc(X)$ be arbitrary. Since coproducts are
  exact in $\Sh(X)$, we can present $F$ as the homotopy
  colimit of its truncations $(\tau_{\leq n}F)_{n \in \bN}$ in
  $\D(X)$, by \cite[Lemma~2.15,
  Prop.~2.16]{bergh-schnuerer-gerbes}.

  Note that all objects $\tau_{\leq n}F$ are in $\D^\bd_\Loc(X)$.
  Hence there are objects $E_n \in \D^\bd(\Loc(X))$ with $E \cong
  \tau_{\leq n} F$ in $\D^\bd_\Loc(X)$.
  Moreover, by fullness (and faithfulness) proved above, each
  transition morphism
  $\tau_{\leq n} F \ra \tau_{\leq n+1} F$ is the image of some
  (unique)
  morphism $E_n \ra E_{n+1}$. The direct system
  $(E_n)_{n \in \bN}$ is a ``direct truncation
  system''
  in the terminology of
  \cite[Section~2.3]{bergh-schnuerer-gerbes} because this is true
  for the system $(\tau_{\leq n}F)_{n \in \bN}$ and our functor
  $\D(\Loc(X)) \ra \D_\Loc(X)$ is t-exact.

  Let $E$ denote the homotopy colimit of
  $(E_n)_{n \in \bN}$
  in $\D(\Loc(X))$. It comes with a triangle
  \begin{equation*}
    \bigoplus_{n \in \bN} E_n
    \xra{1-\text{shift}}
    \bigoplus_{n \in \bN} E_n
    \ra E
    \ra
  \end{equation*}
  in $\D(\Loc(X))$. Its image under the coproduct preserving
  functor
  $\D(\Loc(X)) \ra \D(X)$ presents
  $E$ as the homotopy colimit of the system
  $(E_n)_{n \in \bN}$. Since $F$ is the homotopy colimit of the
  isomorphic system
  $(\tau_{\leq n}F)_{n \in \bN}$, we obtain an isomorphism $F
  \cong E$.
  Hence $F$  is in the essential image of our functor.
\end{proof}

\begin{corollary}
  \label{c:main-one-stratum}
  Let $X$ be a
  locally simply connected,
  locally acyclic
  topological space and assume that
  the equivalent conditions
  \ref{enum:one-stratum-equiv-D+-1}
  and
  \ref{enum:one-stratum-H>0-univ-covering-1}
  in Theorem~\ref{t:main-one-stratum-1} are satisfied.
  Then the following statements are true.
  \begin{enumerate}[label=(\roman*)]
  \item
    \label{enum:RGamma-1}
    The canonical 2-morphism obtained from the universal
    property of right derived functors in the following diagram
    is a 2-isomorphism
    \begin{equation*}
      \xymatrix@R=0pc@C=0.5pc{
        {\D^+(\Loc(X))}
        \ar[dddd]
        \ar[rrrrdd]^-{\Rd_\Loc \Gamma(X,-)}\\
        &&&&\\
        &&&&
        {\D^+(R)}\\
        \\
        {\D^+(X)}
        \ar[rrrruu]_-{\Rd \Gamma(X,-)}
        \ar@{<=}[ruuu]^-{\sim}
      }
    \end{equation*}
    where $\Rd_\Loc \Gamma(X,-)$ denotes the right derived functor
    of $\Gamma(X,-) \colon \Loc(X) \ra \Mod(R)$ and
    $\Rd \Gamma(X,-)$ denotes the right derived functor
    of $\Gamma(X,-) \colon \Sh(X) \ra \Mod(R)$.

    In other words,
    if $I \ra J$ is a quasi-isomorphism from a
    bounded below complex $I$ of injectives in $\Loc(X)$ to a
    bounded below complex $J$ of injectives in $\Sh(X)$, then
    \begin{equation}
      \label{eq:Gamma-I-to-J}
      \Gamma(X,I) \ra \Gamma(X,J)
    \end{equation}
    is a quasi-isomorphism.

    Yet another way to express this is to say that injective objects
    of $\Loc(X)$ are acyclic for $\Gamma(X,-) \colon \Sh(X) \ra
    \Mod(R)$.
  \item
    \label{enum:RHom-1}
    The canonical 2-morphism obtained from the universal
    property of right derived functors in the following diagram
    is a 2-isomorphism
    \begin{equation*}
      \xymatrix@R=0pc@C=0.5pc{
        {\D^+(\Loc(X))^\opp \times \D^+(\Loc(X))}
        \ar[dddd]
        \ar[rrrrrrdd]^-{\Rd_\Loc \Hom(-,-)}\\
        &&&&&&\\
        &&&&&&
        {\D^+(\bZ)}\\
        \\
        {\D^+(X)^\opp \times \D^+(X).}
        \ar[rrrrrruu]_-{\Rd \Hom(-,-)}
        \ar@{<=}[ruuu]^-{\sim}
      }
    \end{equation*}
    In particular, given $L \in \Loc(X)$, then any injective
    object of $\Loc(X)$ is acyclic for
    $\Hom(L,-) \colon \Sh(X) \ra \Mod(\bZ)$, and any
    projective object of $\Loc(X)$ is acyclic for
    $\Hom(-,L) \colon \Sh(X)^\opp \ra \Mod(\bZ)$.
  \end{enumerate}
\end{corollary}

\begin{proof}
  Note that \ref{enum:RGamma-1} is a special case of
  \ref{enum:RHom-1} since $\Hom(\ul{R}_X,-)=\Gamma(X,-)$.  So let
  us prove \ref{enum:RHom-1}. Let $A$ be a bounded below complex in $\Loc(X)$.
  Let $I \ra J$ be a quasi-isomorphism from a bounded below
  complex $I$ of injectives in $\Loc(X)$ to a bounded below
  complex $J$ of injectives in $\Sh(X)$. We need to show that the
  induced morphism
  \begin{equation*}
    \Hom(A,I) \ra \Hom(A,J)
  \end{equation*}
  is a quasi-isomorphism. The $m$-th cohomology of this morphism
  is
  \begin{equation*}
    \Hom_{\K(\Loc(X))}(A,I[m]) \ra \Hom_{\K(X)}(A,J[m])
  \end{equation*}
  which is identified with
  \begin{equation*}
    \Hom_{\D(\Loc(X))}(A,I[m]) \ra \Hom_{\D(X)}(A,J[m])
  \end{equation*}
  by our assumptions on $I$ and $J$. But this morphism factors as
  \begin{equation*}
    \Hom_{\D(\Loc(X))}(A,I[m]) \ra
    \Hom_{\D(X)}(A,I[m]) \ra
    \Hom_{\D(X)}(A,J[m])
  \end{equation*}
  where the first morphism is an isomorphism
  by Theorem~\ref{t:main-one-stratum-2}
  and the second morphism is an isomorphism since $I \ra J$ is an
  isomorphism in $\D(X)$.
\end{proof}

\begin{remark}
  Let $X$ be as in
  Theorem~\ref{t:main-one-stratum-2}
  and assume that $\D^+(\Loc(X)) \ra \D^+_\Loc(X)$
  is an equivalence, i.\,e.\
  condition~\ref{enum:one-stratum-equiv-D+-1} there is satisfied.
  We do not know if the similar functor
  $\D(\Loc(X) \ra \D_\Loc(X)$ is an equivalence as well.
  \end{remark}



\section{Constructible sheaves: Proof of Proposition \ref{p:loc-sc-grothendieck-1} and Theorem \ref{t:equivalent-conditions-for-equivalence-1}}
\label{sec:constr-sheav}

\begin{para}
  \label{para:S-open-in-olS-poset}
  If $(X, \mathcal{S})$ is a stratified space, any stratum $S \in
  \mathcal{S}$ is open in its closure $\ol{S}$ because $S$ is a
  locally closed subset of $X$.
  Given $S, T \in \mathcal{S}$ write $S \leq T$ if $S \subset
  \ol{T}$. Then $\leq$ defines a partial order.
  on
  $\mathcal{S}$. Since $\mathcal{S}$ is finite, there are minimal
  and maximal elements, if $X\not=\emptyset$. In particular,
  there exists a closed
  stratum and an open stratum if $X \not=\emptyset$.
\end{para}

\begin{proposition}
  \label{p:loc-sc-grothendieck} (=Proposition \ref{p:loc-sc-grothendieck-1})
  Let $(X, \mathcal{S})$ be a
  \ref{enum:loc-sc}-stratified space.
  Then $\Cons(X, \mathcal{S})$ is a Grothendieck abelian category
  and the
  inclusion functor
  $\Cons(X, \mathcal{S}) \ra \Sh(X)$ is exact and cocontinuous;
  in   particular, coproducts in $\Sh(X)$ of families of
  objects in $\Cons(X, \mathcal{S})$ are again in $\Cons(X,
  \mathcal{S})$.

  Assume in addition that the stratification is
  \ref{enum:loc-sa}. Then
  $\Cons(X, \mathcal{S})$ is a weak Serre subcategory of
  $\Sh(X)$.  In particular, $\D_\mathcal{S}(X)$ is a thick
  triangulated subcategory of $\D(X)$
  which is closed under
  coproducts in $\D(X)$.
\end{proposition}

\begin{proof}
  If $s \colon S \ra X$ is the inclusion of a stratum,
  then $s^* \colon \Sh(X) \ra \Sh(S)$ is exact and commutes
  with colimits. Using this, all claims are easily deduced
  from   Proposition~\ref{p:Loc-lsc-lsa-1}.
\end{proof}

\begin{para} Assume that $(X, \mathcal{S})$ is $\ref{enum:loc-sc}-\ref{enum:loc-sa}$-stratified space. Then by Proposition \ref{p:loc-sc-grothendieck}
$\Cons(X, \mathcal{S})$ is a Grothendieck abelian category which is a weak Serre subcategory of $\Sh (X)$. Therefore there exists a triangulated category $\D ^+ (\Cons(X, \mathcal{S}))$ and also $\D ^+_{\mathcal{S}}(X)$ is a thick triangulated subcategory of $\D ^+(X)$. We have the obvious triangulated realization functor
\begin{equation}
\label{eq:real}\real :\D ^+ (\Cons(X, \mathcal{S}))\ra \D ^+_{\mathcal{S}}(X)
\end{equation}
and Theorem \ref{t:equivalent-conditions-for-equivalence-1} claims that \eqref{eq:real} is an equivalence under some conditions.

\subsection{Strategy of the proof of Theorem \ref{t:equivalent-conditions-for-equivalence-1}}
We will prove Theorem \ref{t:equivalent-conditions-for-equivalence-1} by induction on the number of strata. The case of one stratum is included in the statement of the theorem. Both categories $\D ^+ (\Cons(X, \mathcal{S}))$  and $\D ^+_{\mathcal{S}}(X)$ can be constructed by gluing the corresponding categories for individual strata. For the induction step we need to check that the realization functor is compatible with the gluing on each side.
Such a compatibility is given by a 2-morphism between functors. We carefully construct these 2-morphisms and check that they are isomorphisms.
\end{para}

\begin{para}
From now on we assume that all stratified spaces are  $\ref{enum:loc-sc}-\ref{enum:loc-sa}$-stratified.
\end{para}



\begin{definition}
  \label{d:stratified-subspace} Let $(X,\mathcal{S})$ be a stratified space.
  A locally closed embedding $e \colon Z \ra X$ is called
  \define{stratified} if its image is a union of elements of
  $\mathcal{S}$. Then $Z$ together with the \define{induced
    stratification}
  $\mathcal{S}_Z:=\{e^{-1}(S) \mid \text{$S \in \mathcal{S}$
    with $S \subset e(Z)$}\}$
  is a stratified space. Clearly it is also $\ref{enum:loc-sc}-\ref{enum:loc-sa}$-stratified.
  A stratified locally closed
  embedding is usually written
  $e \colon (Z, \mathcal{S}_Z) \ra (X, \mathcal{S})$.
\end{definition}

\begin{para}
  Let $e \colon Z \ra X$ be a locally closed embedding of
  topological spaces. Then there are two pairs $(e^*,e_*)$ and
  $(e_!, e^!)$ of adjoint functors
  \begin{align}
    \label{eq:*-adj-1}
    e^* \colon \Sh(X) & \rightleftarrows \Sh(Z) \colon e_*,\\
    \label{eq:!-adj-1}
    e_! \colon \Sh(Z) & \rightleftarrows \Sh(X) \colon e^!.
  \end{align}
  We emphasize that $e^!$ already exists on the abelian level
  for a locally closed embedding.
  The left adjoint functors $e^*$ and $e_!$ are exact, the right
  adjoints $e_*$ and $e^!$ are left exact.
  If $e$ is an open embedding, then $e^!=e^*$ and
  $(e_!, e^!=e^*, e_*)$ is an adjoint triple.
  If $e$ is a closed embedding, then $e_!=e_*$ and
  $(e^*, e_*=e_!, e^!)$ is an adjoint triple.
  In general, we obtain induced adjunctions
  \begin{align}
    \label{eq:*-adj-D-1}
    e^* \colon \D^+(X)
    & \rightleftarrows \D^+(Z) \colon \Rd e_*,\\
    \label{eq:!-adj-D-1}
    e_! \colon \D^+(Z)
    & \rightleftarrows \D^+(X) \colon \Rd e^!
  \end{align}
  on the derived level.
\end{para}

\begin{para}
  \label{p:e^*-e_!-preserve}
  Let $e \colon (Z, \mathcal{S}_Z) \ra (X, \mathcal{S})$ be a
  stratified locally closed embedding. Then the two exact functors
  $e^*$ and
  $e_!$ trivially preserve constructible sheaves and
  constructible complexes (with respect to
  the stratifications $\mathcal{S}$ and $\mathcal{S}_Z$), i.\,e.\
  they induce functors
  \begin{align*}
    e^* \colon \Cons(X, \mathcal{S})
    & \ra \Cons(Z, \mathcal{S}_Z),
    & e_! \colon \Cons(Z, \mathcal{S}_Z)
    & \ra \Cons(X, \mathcal{S}),\\
    e^* \colon \D^+(\Cons(X, \mathcal{S}))
    & \ra \D^+(\Cons(X, \mathcal{S})),
    &  e_! \colon \D^+(\Cons(Z, \mathcal{S}_Z))
    & \ra \D^+(\Cons(X, \mathcal{S})),\\
    e^* \colon \D_\mathcal{S}(X)
    & \ra \D_{\mathcal{S}_Z}(Z),
    & e_! \colon \D_{\mathcal{S}_Z}(Z)
    & \ra \D_\mathcal{S}(X)).
  \end{align*}

 %
  If $d \colon (Y, \mathcal{S}_Y) \ra (Z, \mathcal{S}_Z)$ is
  another stratified locally closed embedding, then obviously
  \begin{equation}
    \label{eq:composition-^*-_!}
    d^* \circ e^* \sira (e \circ d)^*
    \quad \text{and} \quad
    e_! \circ d_! \sira (e \circ d)_!
  \end{equation}
  on the abelian and the triangulated level.
\end{para}

\begin{para}
  \label{para:e^*-e_!-commute-real}
  Let $e \colon (Z, \mathcal{S}_Z) \ra (X, \mathcal{S})$ be a
  stratified locally closed embedding.
  Then the respective identities are 2-isomorphisms
  \begin{equation}
    \label{eq:real-e^*}
    e^* \circ \real
    \xRa[=]{\id}
    \real \circ e^*
  \end{equation}
  between functors from $\D^+(\Cons(X, \mathcal{S}))$ to
  $\D^+(Z)$
  and
  \begin{equation}
    \label{eq:real-e_!}
    e_! \circ \real
    \xRa[=]{\id}
    \real \circ e_!
  \end{equation}
  between functors from $\D^+(\Cons(Z, \mathcal{S}_Z))$ to
  $\D^+(X)$, as illustrated by the commutative diagrams
  \begin{equation*}
    \xymatrix{
      {\D^+(\Cons(X, \mathcal{S}))}
      \ar[r]^-{e^*} \ar[d]^-{\real} &
      {\D^+(\Cons(Z, \mathcal{S}_Z))}
      \ar[d]^-{\real}\\
      {\D^+(X)}
      \ar[r]^-{e^*} &
      {\D^+(Z),}
    }
    \quad
    \xymatrix{
      {\D^+(\Cons(Z, \mathcal{S}_Z))}
      \ar[r]^-{e_!} \ar[d]^-{\real} &
      {\D^+(\Cons(X, \mathcal{S}))}
      \ar[d]^-{\real}\\
      {\D^+(Z)}
      \ar[r]^-{e_!} &
      {\D^+(X)}
    }
  \end{equation*}
\end{para}

\begin{definition}
  A stratification $\mathcal{S}$ of a topological space $X$ is
  called \ref{enum:l_*-preserves-cons} if it satisfies the
  following condition:
    If $l \colon S \ra X$ is the inclusion of any stratum $S \in
    \mathcal{S}$, then
    \textcolor{red}{(cf. \cite[2.1.13, p.~61]{BBD})}
      the right derived functor $\Rd l_* \colon \D^+(S) \ra
      \D^+(X)$
      of $l_* \colon \Sh(S) \ra \Sh(X)$
      maps objects of $\Loc(S)$ (viewed as complexes concentrated
      in degree zero) to
      $\D^+_\mathcal{S}(X)$.
\end{definition}

\begin{lemma}
  \label{t:preserve-cons-1}
  Let $(X, \mathcal{S})$ be a
  \ref{enum:l_*-preserves-cons}-stratified space
  and $e \colon (Z, \mathcal{S}_Z) \ra (X, \mathcal{S})$ a
  stratified locally closed embedding. Then
  the induced stratification $\mathcal{S}_Z$ is
  \ref{enum:l_*-preserves-cons} and in addition to the functors $e^*, e_!$, also the functors
    $e_*, e_!$
  preserve constructible sheaves (where constructibility refers
  to the
  stratifications $\mathcal{S}$ and $\mathcal{S}_Z$,
  respectively). In
  particular,
  the adjunctions
  \eqref{eq:*-adj-1} and
  \eqref{eq:!-adj-1} restrict to adjunctions
  \begin{align}
    \label{eq:*-adj-cons-1}
    e^* \colon \Cons(X, \mathcal{S})
    & \rightleftarrows \Cons(Z, \mathcal{S}_Z) \colon e_*,\\
    \label{eq:!-adj-cons-1}
    e_! \colon \Cons(Z, \mathcal{S}_Z)
    & \rightleftarrows \Cons(X, \mathcal{S}) \colon e^!.
  \end{align}
  Also the four functors
  $e^*, \Rd e^!$,
  $\Rd e_*, e_!$ preserve constructible complexes with bounded
  below cohomology (with respect
  to the
  stratifications $\mathcal{S}$ and $\mathcal{S}_Z$). In
  particular,
  the adjunctions
  \eqref{eq:*-adj-D-1} and
  \eqref{eq:!-adj-D-1} restrict to adjunctions
  \begin{align}
    \label{eq:*-adj-D-cons}
    e^* \colon \D^+_\mathcal{S}(X)
    & \rightleftarrows \D^+_{\mathcal{S}_Z}(Z) \colon \Rd e_*,\\
    \label{eq:!-adj-D-cons}
    e_! \colon \D^+_{\mathcal{S}_Z}(Z)
    & \rightleftarrows \D^+_\mathcal{S}(X) \colon \Rd e^!.
  \end{align}
\end{lemma}

\begin{proof} Let $S\in \mathcal{S}$ be a stratum contained in $Z$ and let $l:S\hookrightarrow Z$ be the locally closed embedding. Then the functor
$$\Rd l_* :\D ^+(S)\to \D ^+(Z)$$
is isomorphic to the composition of functors
$$\Rd l_*\simeq e^*\circ \Rd (e\circ l)_*$$
By our assumption the functor $\Rd (e\circ l)_*$ maps $\D ^+(S)$ to $\D _{\mathcal{S}}(X)$. Also (as mentioned above) the functor $e^*$ maps
$\D _{\mathcal{S}}(X)$ to $\D _{\mathcal{S}_Z}(Z)$. Hence the induced stratification $\mathcal{S}_Z$ is \ref{enum:l_*-preserves-cons}.

The assertions about the exact functors $e^*$ and $e_!$ are obvious and were discussed above. We only consider the functors $e_*$ and $e^!$ and there derived ones.

\medskip
\noindent{\underline{Step 1.}}
The functors $e_*$ and $e^!$ are left exact. Therefore if one proves that
for any $A\in \D^+_{\mathcal{S}_Z}(Z)$ the object $\Rd e_*A$ belongs to the category $\D^+_\mathcal{S}(X)$, this would imply that the functor $e_*$
maps the category $\Cons(Z, \mathcal{S}_Z)$ to $\Cons (X, \mathcal{S})$. Similarly for the functor $e^!$, it suffices to prove that for any $B\in \D^+_\mathcal{S}(X)$, the object $\Rd e^!B$ belongs to $\D^+_{\mathcal{S}_Z}(Z)$. We first consider the functor $\Rd e_*$.

\medskip
\noindent{\underline{Step 2.}} We will proceed by induction on the number of strata in $X$ and in $Z$. Assume that $Z$ consists of one stratum. By assumption for any $F\in \Loc (Z)$ we have $\Rd e_*(F)\in D_{\mathcal{S}}^+(X)$. Since $e_*$ is left exact it follows that $\Rd e_* (A)\in \D^+_\mathcal{S}(X)$

\medskip
\noindent{\underline{Step 3.}} Let $\overline{Z}\subset X$ be the closure of $Z$. Then the embedding $e:Z\hookrightarrow X$ is the composition of stratified locally closed embeddings
\begin{equation}\label{factorization} Z\stackrel{t}{\ra} \overline{Z}\stackrel{\overline{e}}{\ra} X
\end{equation}
where $t$ is an open embedding and $\overline{e}$ is a closed one.
The functor $\overline{e}_*=\overline{e}_!$ is exact and obviously
$$\overline{e}_*\colon \D ^+_{\mathcal{S}_{\overline{Z}}}(\overline{Z})\to D^+_{\mathcal{S}}(X)$$
Hence we may assume that $e\colon Z\to X$ is an open embedding.

\medskip
\noindent{\underline{Step 4.}} Choose a closed (in $Z$) stratum $F\subset Z$. Let
$$F\stackrel{s}{\ra}Z\stackrel{j}{\la}U:=Z-F$$
be the corresponding stratified closed and open embeddings. For $A\in \D ^+_{\mathcal{S}_Z}(Z)$ we have the exact triangle in $\D ^+(Z)$:
\begin{equation}\label{ex-tr-z}
s_*\Rd s^!A\to A\to \Rd j_*j^*A
\end{equation}
There are more strata in $X$ than in $Z$. So by induction on the number of strata we may assume that $\Rd s^!A\in \D _{\mathcal{S}_F}^+(F)$. Hence $s_*\Rd s^!A\in \D ^+_{\mathcal{S}_Z}(Z)$. By the same induction we may assume that $\Rd j_*j^*A\in  \D ^+_{\mathcal{S}_Z}(Z)$. Now using the exact triangle \eqref{ex-tr-z} it suffices to prove that
\begin{equation}\label{remains}\Rd e_*s_*\Rd s^!A,\  \Rd e_*\Rd j_*j^*A\in \D _{\mathcal{S}}^+(X)
\end{equation}
This follows from the isomorphisms
$$\Rd e_*s_*\Rd s^!A=\Rd (es)_*\Rd s^!A,\quad \Rd e_*\Rd j_*j^*A=\Rd (ej)_* j^*A$$
and the induction on the number of strata (this time we compare the number of strata in $Z$, $F$ and $U$).

\medskip
\noindent{\underline{Step 5.}} Now given $B\in \D_{\mathcal{S}}^+(X)$ we want to prove that $\Rd e^!B\in \D _{\mathcal{S}_Z}^+(Z)$. Consider again the factorization \eqref{factorization}. Then $t^!=t^*$, hence we may assume that $e\colon Z\to X$ is a closed embedding. Let $\alpha \colon V:=X-Z\ra X$ be the complementary open embedding. We have the exact triangle in $\D ^+(X)$:
$$e_!\Rd e^!B\ra B\ra \Rd \alpha _*\alpha ^*B$$
Clearly $\alpha ^*B\in \D ^+_{\mathcal{S}_V}(V)$, hence by what we proved above, $\Rd \alpha _*\alpha ^*B\in \D ^+_{\mathcal{S}}(X)$. It follows that
$e_!\Rd e^!B \in \D ^+_{\mathcal{S}}(X)$. But then clearly $\Rd e^!B\in \D ^+_{\mathcal{S}_Z}(Z)$.
\end{proof}

\begin{para} From now on we assume that all stratified spaces are $\ref{enum:loc-sc}-\ref{enum:loc-sa}-
\ref{enum:l_*-preserves-cons}$-stratified.
\end{para}

\begin{lemma}
  \label{l:*-unit-split-epi-on-inj}
  Let $(X, \mathcal{S})$ be a
  stratified space,
  $j \colon (U, \mathcal{S}_U) \ra (X, \mathcal{S})$ a
  stratified open embedding, and $I$
  an injective object of $\Cons(X, \mathcal{S})$.
  Then the adjunction unit $I \ra j_*j^*I$ is a split
  epimorphism of injectives.
\end{lemma}

\begin{proof}
  We use Lemma~\ref{t:preserve-cons-1}.
  Let
  $\eta \colon \id \ra j_*j^*$ be the unit and
  $\epsilon \colon j^*j_* \ra \id$ be the counit of the adjunction
  $j^* \colon \Cons(X, \mathcal{S}) \rightleftarrows \Cons(U,
  \mathcal{S}_U) \colon j_*$.
  Let
  $\zeta \colon \id \ra j^!j_!=j^*j_!$ be the unit and
  $\delta \colon j_!j^!=j_!j^* \ra \id$ be the counit of the
  adjunction
  $j_! \colon \Cons(X, \mathcal{S}) \rightleftarrows \Cons(U,
  \mathcal{S}_U) \colon j^!=j^*$.
  Since the composition in the first row of the following diagram
  is a monomorphism
  and $I$ is injective in $\Cons(X, \mathcal{S})$,
  the dotted arrow $\alpha$ exists making the diagram commutative.
  \begin{equation*}
    \xymatrix{
      {j_!j^*I} \ar[r]^-{\delta_I} \ar[d]_-{\delta_I} &
      {I} \ar[r]^-{\eta_I} &
      {j_*j^*I} \ar@{..>}[lld]^-{\alpha} \\
      {I}
    }
  \end{equation*}
  We claim that $\eta_I \circ \alpha = \id_{j_*j^*I}$.
  We apply $j^*$ to our diagram and get
  \begin{equation*}
    \xymatrix{
      {j^*j_!j^*I} \ar[r]^-{j^*\delta_I}_-{\sim}
      \ar[d]_-{j^*\delta_I}^-{\sim} &
      {j^*I} \ar[r]^-{j^*\eta_I}_-{\sim} &
      {j^*j_*j^*I} \ar@{..>}[lld]^-{j^*\alpha} \\
      {j^*I,}
    }
  \end{equation*}
  where $j^*\delta_I$ and
  $j^*\eta_I$ are isomorphisms; this follows from the commutative
  triangles in the definition of an adjunction and the fact that
  $\zeta$ and $\epsilon$ are 2-isomorphisms.
  Hence $j^*\alpha$ is the inverse of $j^*\eta_I$. This implies
  \begin{equation}
    \label{eq:j^*-eta-alpha}
    j^*(\eta_I \circ \alpha) =j^*\eta_I \circ j^*\alpha =
    \id_{j^*j_*j^*I}.
  \end{equation}
  For arbitrary objects
  $A, B \in \Cons(U, \mathcal{S}_U)$, the first arrow in
  \begin{equation*}
    {\Hom(j_*A, j_*B)} \xra{j^*}
    {\Hom(j^*j_*A, j^*j_*B)} \xra{\epsilon_B \circ ?}
    {\Hom(j^*j_*A, B)}
  \end{equation*}
  is bijective because the composition is the adjunction bijection
  and $\epsilon_B$ is an
  isomorphism.
  Specializing this diagram to $A=B=j^*I$ we see that
  \eqref{eq:j^*-eta-alpha} implies our claim
  $\eta_I \circ \alpha = \id_{j_*j^*I}$. Hence $\eta_I$ is a
  split epimorphism.
\end{proof}

\begin{remark}
  \label{r:*-unit-split-epi-on-inj}
  The proof of Lemma~\ref{l:*-unit-split-epi-on-inj} can easily
  be adapted to show the following statement: Let
  $j \colon U \ra X$ be an open embedding and $I$ an injective
  object of $\Sh(X)$.  Then $I \ra j_*j^*I$ is a split
  epimorphism (it is well-known to be an epimorphism, see e.\,g.\
  the sentence containing equation (1.4.1.1) in \cite{BBD}).
  Slightly more generally, given
  an injective object $I$ of
  an abelian subcategory $\mathcal{X}$ of $\Sh(X)$ such that
  $j_!j^*I$ and $j_*j^*I$ are in $\mathcal{X}$, then
  $I \ra j_*j^*I$ is a split epimorphism.
\end{remark}

\begin{para}
  \label{para:Rcs-functors-e}
  Let $(X, \mathcal{S})$ be a
 stratified space and
  $e \colon (Z, \mathcal{S}_Z) \ra (X, \mathcal{S})$ a
  stratified locally closed embedding.
  Then the categories in \eqref{eq:*-adj-cons-1}
  and
  \eqref{eq:!-adj-cons-1} are Grothendicek abelian, the
  left adjoint functors $e^*$, $e_!$ are exact, and
  their right adjoints
  $e_*$, $e^!$ are left exact.
  If $e$ is an open embedding, then $e^!=e^*$ and
  $(e_!, e^!=e^*, e_*)$ is an adjoint triple.
  If $e$ is a closed embedding, then $e_!=e_*$ and
  $(e^*, e_*=e_!, e^!)$ is an adjoint triple.
  In general, $e_*$ and $e^!$ have
  right derived functors
  \begin{align*}
    \Rd_\cs e_* \colon \D^+(\Cons(Z, \mathcal{S}_Z))
    & \ra \D^+(\Cons(X, \mathcal{S})),\\
    \Rd_\cs e^! \colon \D^+(\Cons(X, \mathcal{S}))
    & \ra \D^+(\Cons(Z, \mathcal{S}_Z)).
  \end{align*}
  The index $\cs$ for \textit{constructible} is used
  in order to avoid confusion with the usual functors
  $\Rd e_*$ and
  $\Rd e^!$ appearing as right adjoints in
  \eqref{eq:*-adj-D-1}
  and \eqref{eq:!-adj-D-1}.
  Moreover, the adjunctions
  \eqref{eq:*-adj-cons-1}
  and
  \eqref{eq:!-adj-cons-1}
  induce adjunctions
  \begin{align}
    \label{eq:*-adj-cons-D}
    e^* \colon \D^+(\Cons(X, \mathcal{S}))
    & \rightleftarrows \D^+(\Cons(Z, \mathcal{S}_Z)) \colon
      \Rd_\cs e_*,\\
    \label{eq:!-adj-cons-D}
    e_! \colon \D^+(\Cons(Z, \mathcal{S}_Z))
    & \rightleftarrows \D^+(\Cons(X, \mathcal{S})) \colon
      \Rd_\cs e^!,
  \end{align}
\end{para}

\begin{lemma}
  \label{l:Rcs-composition-1}
  Let $(X, \mathcal{S})$ be a
  stratified space and let
  $e \colon (Z, \mathcal{S}_Z) \ra (X, \mathcal{S})$ and
  $d \colon (Y, \mathcal{S}_Y) \ra (Z, \mathcal{S}_Z)$ be
  stratified locally closed embeddings.
  Then the canonical 2-morphisms are 2-isomorphisms
  \begin{equation*}
    \Rd_\cs(e \circ d)_* \sira \Rd_\cs e_* \circ \Rd_\cs d_*
    \quad \text{and} \quad
    \Rd_\cs(e \circ d)^! \sira \Rd_\cs d^! \circ \Rd_\cs e^!.
  \end{equation*}
\end{lemma}

\begin{proof}
  Both $e_*$ and $e^!$ preserve injective constructible sheaves
  because they have exact left adjoints
  (Lemma~\ref{t:preserve-cons-1}).
  Therefore the statement follows from the usual isomorphisms
  $(e \circ d)_* \sira e_* \circ d_*$ (even equality)
  and
  $(e \circ d)^! \sira d^! \circ e^!$.
\end{proof}

\begin{para}
  \label{para:sigma-and-tau}
  Let $(X, \mathcal{S})$ be a
  stratified
  space and
  $e \colon (Z, \mathcal{S}_Z) \ra (X, \mathcal{S})$ a
  stratified locally closed embedding.
  Then there is a canonical 2-morphism
  \begin{equation}
    \label{eq:real-e_*-sigma}
    \sigma \colon \real \circ \Rd_\cs e_* \Ra \Rd e_* \circ \real
  \end{equation}
  between functors from $\D^+(\Cons(Z, \mathcal{S}_Z))$ to
  $\D^+(X)$, as illustrated by the diagram
  \begin{equation*}
    \xymatrix{
      {\D^+(\Cons(Z, \mathcal{S}_Z))}
      \ar[r]^-{\Rd_\cs e_*} \ar[d]_-{\real} &
      {\D^+(\Cons(X, \mathcal{S}))}
      \ar[d]^-{\real}
      \ar@{=>}[ld]_-{\sigma}\\
      {\D^+(Z)}
      \ar[r]^-{\Rd e_*} &
      {\D^+(X).}
    }
  \end{equation*}
  The dependency of $\sigma$ on $e$ is suppressed in the notation.
  The construction of $\sigma$ is straightforward
  when one remembers that the right
  derived functor $\Rd_\cs e_*$ comes together with a 2-morphism
  and this pair satisfies a universal property (see e.\,g.\
  \cite[Def.~1.8.1]{KS}).
  Let us explain how $\sigma_A$ is computed, for
  $A \in \D^+(\Cons(Z, \mathcal{S}_Z))$. Let $A \ra I$ be a
  quasi-isomorphism in $\K^+(\Cons(Z, \mathcal{S}_Z))$ where $I$
  is a bounded below complex of injective objects of $\Cons(Z,
  \mathcal{S}_Z)$. Then
  \begin{equation*}
    \Rd_\cs e_* A \sira
    \Rd_\cs e_* I \sila e_*I
  \end{equation*}
  in $\D^+(\Cons(X, \mathcal{S}))$ and then also in
  in $\D^+(X)$.
  Let $A \ra J$ be a
  quasi-isomorphism in $\K^+(Z)$ where $J$
  is a bounded below complex of injective objects of $\Sh(Z)$.
  Then
  \begin{equation*}
    \Rd e_* A \sira
    \Rd e_* J \sila e_*J
  \end{equation*}
  in $\D^+(X)$. Now there is a unique morphism $\kappa \colon I
  \ra J$ in
  $K^+(Z)$ such that the composition $A \ra I \xra{\kappa} J$
  is the quasi-isomorphism $A \ra J$. Clearly $\kappa$ is a
  quasi-isomorphism. Then the morphism
  \begin{equation}
    \label{eq:sigma-A}
    \sigma_A \colon
    \real(\Rd_\cs e_*(A)) \ra
    \Rd e_*(\real(A))
  \end{equation}
  corresponds, via the above isomorphisms, to the morphism
  \begin{equation}
    \label{eq:sigma-A-via-IJ}
    e_*(\kappa) \colon e_*I \ra e_*J.
  \end{equation}
  In particular, if $e$ is a closed embedding, then $\sigma$ is a
  2-isomorphism.
  For general $e$, this is sometimes the case, see
    Theorem~\ref{t:equivalent-conditions-for-equivalence}
  and
  Proposition~\ref{p:sigma-iso-on-strata}.

  Similarly, there is a canonical 2-morphism
  \begin{equation}
    \label{eq:real-e^!-tau}
    \tau \colon \real \circ \Rd_\cs e^! \Ra \Rd e^! \circ \real
  \end{equation}
  between functors from $\D^+(\Cons(X, \mathcal{S}))$ to
  $\D^+(Z)$, as illustrated by the diagram
  \begin{equation*}
    \xymatrix{
      {\D^+(\Cons(X, \mathcal{S}))}
      \ar[r]^-{\Rd_\cs e^!} \ar[d]_-{\real} &
      {\D^+(\Cons(Z, \mathcal{S}_Z))}
      \ar[d]^-{\real} \ar@{=>}[ld]_-{\tau}\\
      {\D^+(X)}
      \ar[r]^-{\Rd e^!} &
      {\D^+(Z).}
    }
  \end{equation*}
  The dependency of $\tau$ on $e$ is suppressed in the notation.
  Given $B \in \D^+(\Cons(X, \mathcal{S}))$
  and quasi-isomorphisms $B \ra I' \xra{\kappa'} J'$ with
  bounded below complexes $I'$ and $J'$ of injective objects of
  $\Cons(X, \mathcal{S})$ and $\Sh(X)$, respectively,
  \begin{equation}
    \label{eq:tau-B}
    \tau_B \colon
    \real(\Rd_\cs e^!(B) \ra \Rd e^!(\real(B))
  \end{equation}
  corresponds, via the obvious isomorphisms, to the morphism
  \begin{equation}
    \label{eq:tau-B-via-I'J'}
    e^!(\kappa') \colon e^!I' \ra e^!J'.
  \end{equation}
  In particular, if $e$ is an open embedding, then $\tau$ is a
  2-isomorphism.
  For general $e$, this is sometimes the case, see
  Theorem~\ref{t:equivalent-conditions-for-equivalence}
  and
  Proposition~\ref{p:sigma-iso-on-strata}.
\end{para}


\begin{lemma}[{Compatibility of unit and counit of
    derived $*$-adjunction with $\real$}]
  \label{l:comp-e^*-e_*}
  Let $(X, \mathcal{S})$ be a
  \ref{enum:loc-sc}-\ref{enum:l_*-preserves-cons}-stratified
  space and
  $e \colon (Z, \mathcal{S}_Z) \ra (X, \mathcal{S})$ a
  stratified locally closed embedding (as in
  \ref{para:sigma-and-tau}).
  Then the diagrams
  \begin{equation*}
    \xymatrix@C=2.5pc{
      {\real(A)} \ar[r]^-{\real(\eta_{\cs,A})} \ar[dd]_-{\id}^-{=} &
      {\real(\Rd_\cs e_*(e^*(A)))}
      \ar[d]_-{\sigma_{e^*A}}^-{\eqref{eq:real-e_*-sigma}}\\
      &
      {\Rd e_*(\real(e^*(A)))}
      \\
      {\real(A)} \ar[r]^-{\eta_{\real(A)}} &
      {\Rd e_*(e^* (\real(A))),}
      \ar[u]^(0.6){\id}_(0.6){=}_(0.4){\eqref{eq:real-e^*}}
    }
    \quad
    \xymatrix@C=2.5pc{
      {\real(e^*(\Rd_\cs e_*(B)))}
      \ar[r]^-{\real(\epsilon_{\cs,B})}_-{\sim}
      &
      {\real(B)} \ar[dd]_-{\id}^-{=}
      \\
      {e^*(\real(\Rd_\cs e_*(B)))}
      \ar[u]^(0.6){\id}_(0.6){=}_(0.4){\eqref{eq:real-e^*}}
      \ar[d]_-{e^*(\sigma_{B})}^-{\eqref{eq:real-e_*-sigma}}
      \\
      {e^*(\Rd e_*(\real(B)))}
      \ar[r]^-{\epsilon_{\real(B)}}_-{\sim}
      &
      {\real(B)}
    }
  \end{equation*}
  commute for all $A \in \D^+(\Cons(X, \mathcal{S}))$
  and $B \in \D^+(\Cons(Z, \mathcal{S}_Z))$ (and are natural in $A$
  and $B$), where $\eta_\cs$ and
  $\epsilon_\cs$ denote unit and counit of the adjunction
  $(e^*, \Rd_\cs e_*)$ and
  $\eta$ and
  $\epsilon$ denote unit and counit of the adjunction
  $(e^*, \Rd e_*)$.
\end{lemma}

\begin{proof}
  Let $e^*A \xra{\alpha} I \xra{\kappa} J$ be quasi-isomorphisms
  where $I$ and $J$ are
  bounded below complexes of injective objects of
  $\Cons(Z, \mathcal{S}_Z)$ and $\Sh(Z)$, respectively.
  Then the square
  \begin{equation*}
    \xymatrix{
      {A} \ar[r] \ar[d]_-{\id}^-{=} &
      {e_*e^*A} \ar[r]^-{e_*(\alpha)} \ar[d]_-{\id}^-{=} &
      {e_*I} \ar[d]^-{e_*(\kappa)}\\
      {A} \ar[r] &
      {e_*e^*A} \ar[r]^-{e_*(\kappa \circ \alpha)} &
      {e_*J}\\
    }
  \end{equation*}
  is commutative and implies that the first diagram in the lemma
  is commutative, by the explicit description
  of $\sigma$ in
  \ref{para:sigma-and-tau} (equation 
  \eqref{eq:sigma-A-via-IJ}).

  Let $B \ra I' \xra{\kappa'} J'$
  be quasi-isomorphisms where $I$ and $J$ are
  bounded below complexes of injective objects of
  $\Cons(Z, \mathcal{S}_Z)$ and $\Sh(Z)$, respectively.
  Then commutativity of the square
  \begin{equation*}
    \xymatrix{
      {e^*e_*I} \ar[r] \ar[d]_-{e^*(e_*(\kappa'))} &
      {I} \ar[d]_-{\kappa'}\\
      {e^*e_*J} \ar[r] &
      {J}
    }
  \end{equation*}
  shows that the second diagram in the lemma is commutative,
  again
  by the explicit description
  of $\sigma$ in
  \ref{para:sigma-and-tau}.
\end{proof}

\begin{lemma}
  [{Compatibility of $*$-adjunction bijections with $\real$}]
  \label{l:*-adj-bijections-real}
  Let $(X, \mathcal{S})$ be a
  \ref{enum:loc-sc}-\ref{enum:l_*-preserves-cons}-stratified space and
  $e \colon (Z, \mathcal{S}_Z) \ra (X, \mathcal{S})$ a
  stratified locally closed embedding (as in
  \ref{para:sigma-and-tau}).
  Let $A \in \D^+(\Cons(X, \mathcal{S}))$
  and $B \in \D^+(\Cons(Z, \mathcal{S}_Z))$. Then the
  diagram
  \begin{equation}
    \label{eq:*-adj-bijections-real}
    \xymatrix{
      {\Hom_{\D(\Cons(Z,\mathcal{S}_Z))}(e^*A, B)}
      \ar[r]^-{\sim}
      \ar[d]^-{\real} &
      {\Hom_{\D(\Cons(X,\mathcal{S}))}(A, \Rd_\cs e_*(B))}
      \ar[d]^-{\real}\\
      {\Hom_{\D(Z)}(\real(e^*A), \real(B))}
      \ar[d]_(0.4){\id}^(0.4){=}^(0.6){\eqref{eq:real-e^*}}
      &
      {\Hom_{\D(X)}(\real(A), \real(\Rd_\cs e_*(B)))}
      \ar[d]_-{? \circ \sigma_B}^-{\eqref{eq:real-e_*-sigma}}
      \\
      {\Hom_{\D(Z)}(e^*(\real(A)), \real(B))}
      \ar[r]^-{\sim} &
      {\Hom_{\D(X)}(\real(A), \Rd e_*(\real(B)))}
    }
  \end{equation}
  is commutative, where the horizontal
  arrows are the adjunction bijections.
\end{lemma}

\begin{proof}
  Let $f \colon e^*A \ra B$ be a morphism in
  $\D^+(\Cons(Z,\mathcal{S}_Z))$.
  Its image in the bottom right corner of our diagram
  via the path through the bottom left corner is the composition
  \begin{equation*}
    \real(A)
    \xra{\eta_{\real(A)}}
    \Rd e_*(e^*(\real(A)))
    \xra{\id}
    \Rd e_*(\real(e^*A))
    \xra{\Rd e_*(\real(f))}
    \Rd e_*(\real(B)),
  \end{equation*}
  and
  its image via the path through the top right corner is the
  composition
  \begin{equation*}
    \real(A)
    \xra{\real(\eta_{\cs,A})}
    \real(\Rd_\cs e_*(e^*A))
    \xra{\real(\Rd_\cs e_*(f))}
    \real(\Rd_\cs e_*(B))
    \xra{\sigma_B}
    \Rd e_*(\real(B)),
  \end{equation*}
  where $\eta$ and $\eta_\cs$ are the adjunction units.
  But these two compositions coincide by
  Lemma~\ref{l:comp-e^*-e_*}
  and the equality
  $\sigma_B \circ \real(\Rd_\cs e_*(f))
  =\Rd e_*(\real(f)) \circ \sigma_{e^*A}$ which holds since
  $\sigma$ is a 2-morphism.
\end{proof}

\begin{lemma}[{Compatibility of unit and counit of
    derived $!$-adjunction} with $\real$]
  \label{l:comp-e_!-e^!}
  Let $(X, \mathcal{S})$ be a
  \ref{enum:loc-sc}-\ref{enum:l_*-preserves-cons}-stratified
  space and
  $e \colon (Z, \mathcal{S}_Z) \ra (X, \mathcal{S})$ a
  stratified locally closed embedding (as in
  \ref{para:sigma-and-tau}).
  Then the diagrams
  \begin{equation*}
    \xymatrix@C=2.5pc{
      {\real(A)} \ar[r]^-{\real(\zeta_{\cs,A})}_-{\sim}
      \ar[dd]_-{\id}^-{=} &
      {\real(\Rd_\cs e^!(e_!(A)))}
      \ar[d]_-{\tau_{e_!A}}^-{\eqref{eq:real-e^!-tau}}\\
      &
      {\Rd e^!(\real(e_!(A)))}
      \\
      {\real(A)} \ar[r]^-{\zeta_{\real(A)}}_-{\sim} &
      {\Rd e^!(e_! (\real(A))),}
      \ar[u]^(0.6){\id}_(0.6){=}_(0.4){\eqref{eq:real-e_!}}
    }
    \quad
    \xymatrix@C=2.5pc{
      {\real(e_!(\Rd_\cs e^!(B)))}
      \ar[r]^-{\real(\delta_{\cs,B})}
      &
      {\real(B)} \ar[dd]_-{\id}^-{=}
      \\
      {e_!(\real(\Rd_\cs e^!(B)))}
      \ar[u]^(0.6){\id}_(0.6){=}_(0.4){\eqref{eq:real-e_!}}
      \ar[d]_-{e_!(\tau_{B})}^-{\eqref{eq:real-e^!-tau}}
      \\
      {e_!(\Rd e^!(\real(B)))}
      \ar[r]^-{\delta_{\real(B)}}
      &
      {\real(B)}
    }
  \end{equation*}
  commute for all
  $A \in \D^+(\Cons(Z, \mathcal{S}_Z))$
  and $B \in \D^+(\Cons(X, \mathcal{S}))$
  (and are natural in $A$
  and $B$), where $\zeta_\cs$ and
  $\delta_\cs$ denote unit and counit of the adjunction
  $(e_!, \Rd_\cs e^!)$ and
  $\zeta$ and
  $\delta$ denote unit and counit of the adjunction
  $(e_!, \Rd e^!)$.
\end{lemma}

\begin{proof}
  Formally the same as that of Lemma~\ref{l:comp-e^*-e_*};
  use
  $(e_!,e^!)$ instead of $(e^*,e_*)$ and
  the explicit description
  of $\tau$ in
  \ref{para:sigma-and-tau} (equation 
  \eqref{eq:tau-B-via-I'J'}).
\end{proof}

The next lemma is strictly speaking not needed for us.

\begin{lemma}
  [{Compatibility of $!$-adjunction bijections with $\real$}]
  \label{l:!-adj-bijections-real}
    Let $(X, \mathcal{S})$ be a
  \ref{enum:loc-sc}-\ref{enum:l_*-preserves-cons}-stratified space and
  $e \colon (Z, \mathcal{S}_Z) \ra (X, \mathcal{S})$ a
  stratified locally closed embedding (as in
  \ref{para:sigma-and-tau}).
  Let $A \in \D^+(\Cons(Z, \mathcal{S}_Z))$
  and $B \in \D^+(\Cons(X, \mathcal{S}))$. Then the
  diagram
  \begin{equation}
    \label{eq:!-adj-bijections-real}
    \xymatrix{
      {\Hom_{\D(\Cons(X,\mathcal{S}))}(e_!A, B)}
      \ar[r]^-{\sim}
      \ar[d]^-{\real} &
      {\Hom_{\D(\Cons(Z,\mathcal{S}_Z))}(A, \Rd_\cs e^!(B))}
      \ar[d]^-{\real}\\
      {\Hom_{\D(X)}(\real(e_!A), \real(B))}
      \ar[d]_(0.4){\id}^(0.4){=}^(0.6){\eqref{eq:real-e_!}}
      &
      {\Hom_{\D(Z)}(\real(A), \real(\Rd_\cs e^!(B)))}
      \ar[d]_-{\tau_B \circ ?}^-{\eqref{eq:real-e^!-tau}}
      \\
      {\Hom_{\D(X)}(e_!(\real(A)), \real(B))}
      \ar[r]^-{\sim} &
      {\Hom_{\D(Z)}(\real(A), \Rd e^!(\real(B)))}
    }
  \end{equation}
  is commutative, where the horizontal
  arrows are the adjunction bijections.
\end{lemma}

\begin{proof}
  This is proven as Lemma~\ref{l:*-adj-bijections-real}
  using
  Lemma~\ref{l:comp-e_!-e^!}.
\end{proof}

\begin{para}
  \label{para:open-closed-Sh}
  The following material is essentially \cite[1.4.1]{BBD}. We included it for a reference in \ref{para:open-closed-Cons}.
  Let $X$ be a topological space with an open subspace $U$ and
  its closed complement $F$. Let $F \xra{i} X \xla{j} U$ be the
  inclusions. Then we have the following six functors
  \begin{equation}
    \label{eq:recollement-Sh}
    \xymatrix@=45pt{
      {\Sh(F)} \ar[r]|{i_*=i_!} &
      \ar@/_1pc/[l]_{i^*} \ar@/^1pc/[l]^{i^!}
      {\Sh(X)}
      \ar[r]|{j^!=j^*} &
      \ar@/_1pc/[l]_{j_!} \ar@/^1pc/[l]^{j_*}
      {\Sh(U)}
    }
  \end{equation}
  where each functor is left adjoint to the functor directly
  below it, i.\,e.\ there are four adjunctions $(i^*, i_*)$,
  $(i_!,i^!)$, $(j^*, j_*)$, $(j_!, j^!)$,
  the four ``upper'' functors $i^*$, $i_*=i_!$, $j_!$, $j^!=j^*$
  are exact, and the right adjoints $i^!$ and $j_*$ are left
  exact.
  The horizontal composition of the two top (resp.\
  middle resp.\ bottom) functors is zero, i.\,e.\ $i^*j_!=0$,
  $j^*i_*=0$, $i^!j_*=0$.
  The functors $i_*=i_!$, $j_!$ and $j_*$ are fully faithful;
  this means that the four adjunction morphisms
  $i^*i_* \sira \id \sira i^!i_!$,
  $j^* j_* \sira \id \sira j^! j_!$ are isomorphisms.
  Moreover, the adjunction morphisms fit into exact sequences
  \begin{align*}
    0 & \ra j_!j^! A=j_!j^*A \ra A \ra i_* i^*A \ra 0,\\
    0 & \ra i_! i^! A=i_*i^!A \ra A \ra j_*j^* A,
  \end{align*}
  for each $A$ in $\Sh(X)$.
   The last morphism in the second
  sequence is a (split) epimorphism if $A$ is injective (or flabby) in
  $\Sh(X)$.
\end{para}

\begin{para}
  \label{para:open-closed-D-Sh}
  Keep the setting of \ref{para:open-closed-Sh}. Then
  \eqref{eq:recollement-Sh} induces a similar diagram
  \begin{equation}
    \label{eq:recollement-D-Sh}
    \xymatrix@=45pt{
      {\D^+(F)} \ar[r]|{i_*=i_!} &
      \ar@/_1pc/[l]_{i^*} \ar@/^1pc/[l]^{\Rd i^!}
      {\D^+(X)}
      \ar[r]|{j^!=j^*} &
      \ar@/_1pc/[l]_{j_!} \ar@/^1pc/[l]^{\Rd j_*}
      {\D^+(U),}
    }
  \end{equation}
  of derived
  categories and functors,
  where $i^*$, $i_*=i_!$, $j_!$, $j^!=j^*$ are trivially induced
  from the corresponding exact functors and $\Rd i^!$, $\Rd j_*$
  are the right derived functors of the left exact functors $i^!$
  and $j_*$. There are induced adjunctions
  $(i^*, i_*)$,
  $(i_!,\Rd i^!)$, $(j^*, \Rd j_*)$, $(j_!, j^!)$.
  The horizontal composition of the two top (resp.\
  middle resp.\ bottom) functors is zero, i.\,e.\ $i^*j_!=0$,
  $j^*i_*=0$, $\Rd i^! \Rd j_*=0$ (for the last equality use that
  $j_* \colon \Sh(U) \ra \Sh(X)$ preserves injectives because its
  left adjoint $j^*$ is exact).
  The functors $i_*=i_!$, $j_!$ and $\Rd j_*$ are fully faithful
  because the four adjunction morphisms
  $i^*i_* \sira \id \sira (\Rd i^!) i_!$,
  $j^* (\Rd j_*) \sira \id \sira j^! j_!$ are isomorphisms (note
  that $i_! \colon \Sh(F) \ra \Sh(X)$ preserves injectives
  because its left adjoint $i^*$ is exact).
  Moreover, the adjunction morphisms can be completed by unique
  morphisms $d$ and $d'$ into triangles
  \begin{align}
    \label{eq:ji-standard-triangle-Sh}
    j_!j^* A & \ra A \ra i_* i^*A \xra{d} [1] j_!j^* A,\\
    \label{eq:ij-standard-triangle-Sh}
    i_* (\Rd i^! A) & \ra A \ra \Rd j_*(j^* A)
    \xra{d'} [1] i_*(\Rd i^! A),
  \end{align}
  called \textit{standard triangles},
  for each $A$ in $\D^+(X)$ (see \cite[(1.4.3.4)]{BBD}).
\end{para}

\begin{para}
  \label{para:open-closed-Cons}
  Let $(X, \mathcal{S})$ be a
  \ref{enum:loc-sc}-\ref{enum:l_*-preserves-cons}-stratified space,
  $j \colon (U, \mathcal{S}_U) \ra (X, \mathcal{S})$ a stratified
  open embedding and $i \colon
  (F, \mathcal{S}_F) \ra (X, \mathcal{S})$ the complementary
  stratified closed embedding.
  Then all statements
  of \ref{para:open-closed-Sh}
  with $\Sh(-)$ replaced by $\Cons(-,\mathcal{S}_{(-)})$ are
  true. This follows from
  Proposition~\ref{p:loc-sc-grothendieck},
  Lemma~\ref{t:preserve-cons-1} and
  Lemma~\ref{l:*-unit-split-epi-on-inj}.
  Similarly, all statements of \ref{para:open-closed-D-Sh}
  with $\D^+(-)$ replaced by
  $\D^+(\Cons(-,\mathcal{S}_{-}))$ and
  $\Rd$ replaced by $\Rd_\cs$ are true.
  In particular, we have \textit{standard triangles}
  \begin{align}
    \label{eq:ji-standard-triangle-Cons}
    j_!j^* A & \ra A \ra i_* i^*A \ra [1] j_!j^* A,\\
    \label{eq:ij-standard-triangle-Cons}
    i_* (\Rd_\cs i^! A) & \ra A \ra \Rd_\cs j_*(j^* A)
    \ra [1] i_*(\Rd_\cs i^! A),
  \end{align}
  for each $A$ in $\D^+(\Cons(X, \mathcal{S}))$.
  Let us compare these two triangles with the triangles
  \eqref{eq:ji-standard-triangle-Sh}
  and \eqref{eq:ij-standard-triangle-Sh}.

  On the one hand,
  the image of
  \eqref{eq:ji-standard-triangle-Cons} under the functor
  $\real$ coincides with the triangle
  \eqref{eq:ji-standard-triangle-Sh} for $\real(A)$,
  (by the identities
  \eqref{eq:real-e^*} and \eqref{eq:real-e_!})
  i.\,e.\
  the identity morphisms define an isomorphism
  \begin{equation}
    \label{eq:ji-triangle-compatible}
    \xymatrix{
      {\real(j_!j^* A)} \ar[r] &
      {\real(A)} \ar[r] &
      {\real(i_* i^*A)} \ar[r] &
      {\Sigma \real(j_!j^* A)} \\
      {j_!j^* \real(A)} \ar[r]
      \ar[u]^-{\id}_-{=}
      &
      {\real(A)} \ar[r]
      \ar[u]^-{\id}_-{=} &
      {i_* i^*\real(A)} \ar[r]
      \ar[u]^-{\id}_-{=}
      &
      {\Sigma j_!j^* \real(A)}
      \ar[u]^-{\id}_-{=}
    }
  \end{equation}
  of triangles, for any $A$ in $\D^+(\Cons(X, \mathcal{S}))$.

  On the other hand,
  the image of
  \eqref{eq:ij-standard-triangle-Cons} under the functor
  $\real$ and triangle
  \eqref{eq:ij-standard-triangle-Sh} for $\real(A)$ can be
  compared as follows. We claim that
  \begin{equation}
    \label{eq:compatibility-ij-triangles}
    \xymatrix@C=2.5pc{
      {\real(i_*(\Rd_\cs i^! A))}
      \ar[r]^-{\real(\delta_{\cs,A})} \ar@{..>}[d]_{i_*(\tau_A)}
      \ar@{}[rd]|-{\text{\ding{192}}} &
      {\real(A)} \ar[r]^-{\real(\eta_{\cs,A})} \ar[d]_-{\id}^-{=}
      \ar@{}[rd]|-{\text{\ding{193}}} &
      {\real(\Rd_\cs j_*(j^* A))}
      \ar[r] \ar @{..>}[d]_-{\sigma_{j^* A}} &
      {\Sigma\real(i_*(\Rd_\cs i^! A))}
      \ar[d]_-{\Sigma i_*(\tau_A)}\\
      {i_* (\Rd i^! (\real(A)))} \ar[r]^-{\delta_{\real(A)}} &
      {\real(A)} \ar[r]^-{\eta_{\real(A)}} &
      {\Rd j_*(j^* (\real(A)))} \ar[r] &
      {\Sigma i_* (\Rd i^! (\real(A)))}
    }
  \end{equation}
  is a morphism of triangles.  Indeed, first note that the dotted
  arrow $i_*(\tau_A)$ makes the square \ding{192} commutative,
  by Lemma~\ref{l:comp-e_!-e^!}, and that the dotted arrow
  $\sigma_{j^* A}$ makes the square \ding{193} commutative, by
  Lemma~\ref{l:comp-e^*-e_*}.  Second, $\real$ commutes with
  $i_*$ (see \eqref{eq:real-e_!}), and every morphism from the
  essential image of $i_*$ to the essential image of $\Rd j_*$
  is zero (use $j^* i_*=0$).  Hence any morphism from
  $\real(i_*(\Rd_\cs i^! A))$ to any shift of
  $\Rd j_*(j^* (\real(A)))$ is zero.  Therefore
  \cite[Prop.~1.1.9]{BBD} shows that there is no other choice for
  the left (resp. right) dotted arrow making the square
  \ding{192} (resp. \ding{193}) commutative, and that the triple
  $(i_*(\tau_A), \id, \sigma_{j^* A})$ in fact is a morphism of
  triangles as claimed.
\end{para}

\begin{proposition}
  \label{p:sigma-iso-on-strata}



  Assume that
  \begin{equation*}
    \sigma \colon \real \circ \Rd_\cs s_*
    \xRa{\eqref{eq:real-e_*-sigma}} \Rd s_* \circ
    \real
  \end{equation*}
  is an isomorphism for each
  $S \in \mathcal{S}$, where $s \colon S \ra X$ is the inclusion.
  Let
  $e \colon (Z, \mathcal{S}_Z) \ra (X, \mathcal{S})$ be a
  stratified locally closed embedding.
  Then the 2-morphisms $\sigma$ and $\tau$ explained in
  \ref{para:sigma-and-tau} (see
  \eqref{eq:real-e_*-sigma}
  and
  \eqref{eq:real-e^!-tau})
  are 2-isomorphisms
  \begin{align*}
    \sigma \colon \real \circ \Rd_\cs e_*
    & \siRa \Rd e_* \circ \real,\\
    \tau \colon \real \circ \Rd_\cs e^!
    & \siRa \Rd e^! \circ \real.
  \end{align*}
  In particular,
  if $j \colon (U, \mathcal{S}_U) \ra (X, \mathcal{S})$
  is a stratified open embedding
  and $i \colon (F, \mathcal{S}_F) \ra (X, \mathcal{S})$ is its
  complementary stratified closed embedding, then
  the morphism of triangles~\eqref{eq:compatibility-ij-triangles}
  in \ref{para:open-closed-Cons}
  is an isomorphism.
\end{proposition}

\begin{proof} We proceed in several steps.

\medskip
  \noindent\underline{Step 1.} (Reduction to open and closed embeddings)
  Let $j \colon (U, \mathcal{S}_U) \ra (X, \mathcal{S})$
  be a stratified open embedding
  and $i \colon (F, \mathcal{S}_F) \ra (X, \mathcal{S})$ its
  complementary stratified closed embedding.
  We claim that it is sufficient to show that the
  two 2-morphisms
  \begin{align}
    \label{eq:real-j_*-sigma-iso}
    \sigma \colon \real \circ \Rd_\cs j_*
    & \Ra \Rd j_* \circ \real,\\
    \label{eq:real-i^!-tau-iso}
    \tau \colon \real \circ \Rd_\cs i^!
    & \Ra \Rd i^! \circ \real
  \end{align}
  are 2-isomorphisms.

  Indeed, 
  the
  locally closed embedding $e$ factors as a stratified open
  embedding $f \colon (Z,
  \mathcal{S}_Z) \ra (\ol{Z}, \mathcal{S}_{\ol{Z}})$ (with dense
  image) followed by a
  stratified closed embedding
  $g \colon (\ol{Z}, \mathcal{S}_{\ol{Z}}) \ra (X, \mathcal{S})$.
  Then
  $\Rd_\cs e_* = \Rd_\cs g_* \circ \Rd_\cs f_*$
  and
  $\Rd_\cs e^! = \Rd_\cs f^! \circ \Rd_\cs g^!$
  (Lemma~\ref{l:Rcs-composition-1}),
  and we already know that $\sigma$ is a 2-isomorphism for
  stratified
  closed
  embeddings (see sentence after
  \eqref{eq:sigma-A-via-IJ}) and that $\tau$ is a
  2-isomorphism for stratified open embeddings (see sentence after
  \eqref{eq:tau-B-via-I'J'}). Now the claim is clear.

\medskip
  \noindent\underline{Step 2.}
  We claim that \eqref{eq:real-i^!-tau-iso}
  is a 2-isomorphism if and only if
  \eqref{eq:real-j_*-sigma-iso} is a 2-isomorphism.

  Indeed, this follows from the morphism of triangles
  \eqref{eq:compatibility-ij-triangles}
  in \ref{para:open-closed-Cons}
  together with the fact that $j^* \colon \D^+(\Cons(X,
  \mathcal{S}) \ra D^+(\Cons(U, \mathcal{S}_U))$ is essentially
  surjective and that $i_* \colon D^+(\Cons(F, \mathcal{S}_F))
  \ra
  D^+(\Cons(X, \mathcal{S}))$ is fully faithful.

  \medskip
  \noindent\underline{Step 3.}
  By the above steps,
  it suffices to show that \eqref{eq:real-j_*-sigma-iso} is
  a 2-isomorphism. We do this
  by induction on the number of strata in $U$, the case
  $U=\emptyset$ being trivial, and the base case that $U$
  consists of precisely one (open) stratum being clear by
  assumption.


  Now assume that $U$ consists of $\geq 2$ strata.
  Let $\kappa \colon I \ra J$ be a quasi-isomorphism in $\C(U)$
  where $I$ and $J$ are bounded below complexes of injective
  objects of $\Cons(U, \mathcal{S}_U)$ and $\Sh(U)$,
  respectively.
  By the explicit description
  of
  $\sigma$
  in
  \ref{para:sigma-and-tau} (see equations~\eqref{eq:sigma-A} and
  \eqref{eq:sigma-A-via-IJ}),
  we need to show that $j_*(\kappa) \colon j_*(I) \ra j_*(J)$ is
  a quasi-isomorphism.

  Let $E$ be a non-empty proper closed subset of $U$ that is a
  union of strata (for example a stratum that is closed in $U$;
  existence follows from \ref{para:S-open-in-olS-poset}). Let
  $V:=U-E$
  be its (non-empty) open complement in $U$.
  We obtain the commutative diagram
  \begin{equation*}
    \xymatrix{
      {X-V} \ar[r]^-{d} &
      {X}\\
      {E} \ar[u]^-{k} \ar[r]^-{e} &
      {U} \ar[u]^-{j} &
      {V} \ar[l]_-{v}
    }
  \end{equation*}
  where the square is cartesian, $v$, $j$ and $k$ are stratified
  open embeddings, and $d$ and $e$ are stratified closed
  embeddings.

  Consider the commutative diagram
  \begin{equation*}
    \xymatrix@C=3pc{
      {0} \ar[r] &
      {e_* e^! I}
      \ar[r] \ar[d]_-{e_*(e^!(\kappa))}
      &
      {I} \ar[r] \ar[d]_-{\kappa}
      &
      {v_* v^* I}
      \ar[d]_-{v_*(v^*(\kappa))}
      \ar[r]
      &
      {0}
      \\
      {0} \ar[r] &
      {e_* e^! J} \ar[r] &
      {J} \ar[r] &
      {v_*v^* J}
      \ar[r]
      &
      {0}
    }
  \end{equation*}
  in $\C^+(U)$ (the image of this diagram in $\D^+(U)$ is
  essentially the diagram \eqref{eq:compatibility-ij-triangles}
  for the closed-open decomposition $E \xra{e} U \xla{v} V$).
  Note that the top row of this diagram is a levelwise split
  short exact sequence of complexes of injective objects of
  $\Cons(U, \mathcal{S}_U)$, by
  Lemma~\ref{l:*-unit-split-epi-on-inj} and the fact that all
  functors $e_*$, $e^!$, $v^*$, $v_*$ have exact left
  adjoints.
  Similarly, using
  Remark~\ref{r:*-unit-split-epi-on-inj}, the bottom row is a
  levelwise split short exact sequence of complexes of
  injective objects of $\Sh(U)$.

  If we apply $\kappa$ to this diagram, we again obtain a
  morphism of (levelwise split) short exact sequences.
  The associated long exact cohomology sequence then shows that
  $j^*(\kappa)$ is a quasi-isomorphism if both morphisms
  \begin{align}
    \label{eq:e-kappa}
    j_*(e_*(e^!(\kappa))) \colon
    j_*e_* e^! I \ra
    j_*e_* e^! J,\\
    \label{eq:v-kappa}
    j_*(v_*(v^*(\kappa)))
    \colon
    j_*v_* v^* I
    \ra
    j_*v_*v^* J
  \end{align}
  are quasi-isomorphisms.

  Since $v^*(\kappa) \colon v^* I \ra v^* J$ is a
  quasi-isomorphism between bounded below complexes of injectives
  in $\Cons(V, \mathcal{S}_V)$ and $\Sh(V)$, respectively, its
  $*$-pushforward under the stratified open embdding $j \circ v$
  is a quasi-isomorphism by the induction assumption since $V$
  consists of strictly less strata than $U$.
  This shows that \eqref{eq:v-kappa} is a quasi-isomorphism.

  Again using that $V$ has less strata than $U$, the induction
  assumption shows that
  $\real \circ \Rd_\cs v_*
  \Ra \Rd v_* \circ \real$ is a 2-isomorphism. Hence Step 2 shows
  that
  $e^!(\kappa) \colon
  e^! I \ra
  e^! J$ is a quasi-isomorphism. Note that $e^!I$ and $e^!J$ are
  bounded below complexes of injective
  objects of $\Cons(E, \mathcal{S}_E)$ and $\Sh(E)$,
  respectively. Hence the induction assumption for the
  stratified open embedding $k$, using that $E$ has strictly less
  strata than $U$, shows that $k_*(e^!(\kappa))$ is a
  quasi-isomorphism. Since $d$ is a closed embedding,
  $d_*(k_*(e^!(\kappa)))$ is a
  quasi-isomorphism. Since $d_* \circ k_* \cong j_* \circ
  e_*$ we see that
  \eqref{eq:e-kappa} is a quasi-isomorphism. This establishes the
  induction step.
\end{proof}

\subsection{Completion of the proof of Theorem \ref{t:equivalent-conditions-for-equivalence-1}}

\begin{theorem}
  \label{t:equivalent-conditions-for-equivalence}(=Theorem \ref{t:equivalent-conditions-for-equivalence-1})
  For a
  \ref{enum:loc-sc}-\ref{enum:loc-sa}-\ref{enum:l_*-preserves-cons}-stratified
  space $(X, \mathcal{S})$, the following conditions are equivalent.
  \begin{enumerate}
  \item
    \label{enum:equivalence-X}
    The functor
    \begin{equation}
      \label{eq:equi-D+-X}
      \real \colon \D^+(\Cons(X, \mathcal{S})) \ra
      \D_\mathcal{S}^+(X)
    \end{equation}
    is an equivalence.
  \item
    \label{enum:equiv-S-and-sigma-iso}
    For all strata $S \in \mathcal{S}$,
    the functor
    \begin{equation}
      \label{eq:equi-D+-S}
      \real \colon \D^+(\Loc(S)) \ra
      \D_\Loc^+(S)
    \end{equation}
    is an equivalence (cf.\
    Remark~\ref{r:stratawise-equiv}),
    and
    \begin{equation}
      \label{eq:equi-D+-sigma}
      \sigma \colon \real \circ \Rd_\cs s_*
      \xRa{\eqref{eq:real-e_*-sigma}} \Rd s_* \circ
      \real
    \end{equation}
    is a 2-isomorphism, where $s
    \colon S \ra X$ denotes the inclusion (cf.\
    Proposition~\ref{p:sigma-iso-on-strata}).
  \end{enumerate}
  If these equivalent conditions are satisfied, then
  \begin{equation}
    \label{eq:equi-D+-Z}
    \real \colon \D^+(\Cons(Z, \mathcal{S}_Z)) \ra
    \D_{\mathcal{S}_Z}^+(Z)
  \end{equation}
  is an equivalence for all stratified locally closed embeddings
  $e \colon (Z, \mathcal{S}_Z) \ra (X, \mathcal{S})$,
\end{theorem}

\begin{remark}
  \label{r:stratawise-equiv}
  If all strata are
  locally simply connected and
  locally singular-acyclic,
  then the condition that
  \eqref{eq:equi-D+-S} is an equivalence, for $S \in \mathcal{S}$,
  is equivalent to the condition
  that the universal covering of each path component of $S$ is
  acyclic (see Theorem~\ref{t:main-one-stratum-2}).
  The condition that \eqref{eq:equi-D+-sigma}
  is a 2-isomorphism will be analyzed later on.
\end{remark}

\begin{proof}
  Assume that \ref{enum:equivalence-X} holds.
  Let $e \colon (Z, \mathcal{S}_Z) \ra (X, \mathcal{S})$
  be a stratified locally closed embedding and consider the
  commutative diagram
  \begin{equation}
    \label{eq:Z-X-real}
    \xymatrix{
      {\D^+(\Cons(Z, \mathcal{S}_Z))}
      \ar[r]^-{e_!} \ar[d]^-{\real_Z} &
      {\D^+(\Cons(X, \mathcal{S}))}
      \ar[d]^-{\real_X}\\
      {\D^+_{\mathcal{S}_Z}(Z)}
      \ar[r]^-{e_!} &
      {\D^+_\mathcal{S}(X)}
    }
  \end{equation}
  defined by the exact functor $e_!$ (\ref{para:e^*-e_!-commute-real}).
  We claim that $\real_Z$ is an equivalence.
  Note that both horizontal functors are
  fully faithful, because $e$ factors as the
  stratified open
  embedding $j
  \colon Z
  \ra \ol{Z}$
  followed by the stratified closed embedding $i \colon \ol{Z}
  \ra X$, we have $e_! =i_! \circ j_!$ and $i_!$ and $j_!$
  are
  fully faithful on the level of derived categories of
  constructible sheaves and on the level of derived categories of
  sheaves, by \ref{para:open-closed-D-Sh} and
  \ref{para:open-closed-Cons}.
  Since $\real =\real_X$ is an
  equivalence by assumption, we deduce that $\real_Z$ is fully
  faithful. Let $B \in \D^+_{\mathcal{S}_Z}(Z)$. Since $\real_X$
  is essentially
  surjective, there are an object $A \in \D^+(\Cons(X,
  \mathcal{S}))$ and an
  isomorphism $\real(A) \cong e_!(B)$.
  Then
  \begin{equation*}
    \real_Z(e^* A)
    \xla[=]{\eqref{eq:real-e^*}}
    e^*\real(A)
    \cong
    e^*e_!(B)
    \cong j^* i^* i_* j_!(B)
    \sira j^* j_!(B)
    \sila
    B.
  \end{equation*}
  This shows that $\real_Z$ is essentially surjective and hence
  an equivalence. 

  Since the vertical arrows in \eqref{eq:Z-X-real} are
  equivalences, the right adjoint $\Rd_\cs e^!$ of the upper
  horizontal functor $e_!$ ``coincides'' with the right adjoint
  $\Rd e^!$ of
  the lower horizontal functor $e_!$, and hence
  $\tau$ (\ref{para:sigma-and-tau})
  is a 2-isomorphism.
  Let us turn this into a formal argument using diagram
  \eqref{eq:!-adj-bijections-real} in
  Lemma~\ref{l:!-adj-bijections-real}. In our setting, all arrows
  there except for the arrow labeled $? \circ \tau_B$ are
  bijective. Hence,
  using the Yoneda lemma and the fact that
  $\real=\real_Z$ is an equivalence shows
  that $\tau_B$ is an isomorphism.

  The same argument with $e_!$ works also for $e_*$ and shows
  that $\sigma$ is a 2-isomorphism.

  This shows that \ref{enum:equivalence-X} implies that
  \eqref{eq:equi-D+-Z} is an equivalence
  and that
  $\sigma$ and $\tau$ are 2-isomorphisms.

  In particular,
  \ref{enum:equivalence-X}
  implies \ref{enum:equiv-S-and-sigma-iso}, because
  $(S, \{S\}) \ra (X, \mathcal{S})$ is a stratified locally
  closed embedding, for all strata $S \in \mathcal{S}$.

  Conversely, assume that
  \ref{enum:equiv-S-and-sigma-iso} holds.
  We want to show that \eqref{eq:equi-D+-X}
  is an equivalence.

  Given strata $S, T \in \mathcal{S}$, let $s \colon S \ra X$ and
  $t \colon T \ra X$ denote the corresponding stratified locally
  closed embeddings.  Consider the classes
  \begin{align*}
    \mathcal{A}
    & :=\{s_!(A) \mid S \in \mathcal{S}, A \in
      \D^+(\Loc(S))=\D^+(\Cons(S, \mathcal{S}_S))\},\\
    \mathcal{B}
    & :=\{\Rd_\cs t_*(B) \mid T \in \mathcal{S}, B \in
      \D^+(\Loc(T))=\D^+(\Cons(T, \mathcal{S}_T))\}
  \end{align*}
  of objects of $\D^+(\Cons(X, \mathcal{S}))$.
  In order to prove that
  \begin{equation}
    \label{eq:real-main}
    \real \colon \D^+(\Cons(X, \mathcal{S}) \ra \D(X)
  \end{equation}
  is fully faithful it suffices,
  by the obvious d\'evissage using
  standard triangles of the form
  \eqref{eq:ji-standard-triangle-Cons}
  and \eqref{eq:ij-standard-triangle-Cons},

  to show that
  \begin{equation*}
    \real \colon
    \Hom_{\D(\Cons(X, \mathcal{S}))}(s_!(A), \Rd_\cs t_*(B))
    \ra
    \Hom_{\D(X)}(\real(s_!(A)), \real(\Rd_\cs t_*(B)))
  \end{equation*}
  is bijective, for all $S, T \in \mathcal{S}$ and $A, B \in
  \D^+(\Cons(X, \mathcal{S}))$.
  Lemma~\ref{l:*-adj-bijections-real}
  together with
  our assumption
  that $\sigma_B \colon \real(\Rd_\cs t_*(B))
  \sira \Rd t_*(\real(B))$
  is an isomorphism show that it
  is sufficient to see that
  \begin{equation*}
    \real \colon
    \Hom_{\D(\Loc(T))}(t^*(s_!(A)), B)
    \ra
    \Hom_{\D(T)}(\real(t^*(s_! A)), \real(B))
  \end{equation*}
  is bijective.
  But this follows from the assumption that
  the functor \eqref{eq:equi-D+-S}
  is fully faithful.

  Clearly,
  \eqref{eq:real-main} lands in
  $\D_\mathcal{S}^+(X)$, so it remains to show that
  $\D_\mathcal{S}^+(X)$
  is the essential image of
  \eqref{eq:real-main}.

  D\'evissage using triangles of the form
  \eqref{eq:ji-standard-triangle-Sh}
  shows
  that it is enough to prove that all objects of the form
  $s_!(A)$, for $S \in \mathcal{S}$ and
  $A \in \D^+_{\mathcal{S}_S}(S)=\D^+_\Loc(S)$
  are in the
  essential image of \eqref{eq:real-main}. But this is clear by
  the assumption that \eqref{eq:equi-D+-S} is essentially
  surjective and the fact that
  $e_! \circ \real = \real \circ e_!$ (see \eqref{eq:real-e_!})
  This show that
  \eqref{eq:equi-D+-X} is an equivalence.
\end{proof}




\section{Stratified spaces with normal structure}

We would like to take a closer look at the equivalent conditions in Theorem
\ref{t:equivalent-conditions-for-equivalence-1} in case of stratified spaces with {\it normal structure} to be defined below. For simplicity we limit our discussion to the case when each stratum is a manifold.

\begin{definition} A {\bf cone} over a topological space $L$ is the space $$\Cone(L)=\frac{L \times [0,1)}{L \times \{0\}}$$
\end{definition}

\begin{definition}\label{def-normal-str} Let $(X,\mathcal{S})$ be a stratified space. We say that it has a {\bf normal structure} if the following holds
\begin{itemize}
\item Each stratum is a connected manifold and there are finitely many strata.
\item For each stratum $T\in \mathcal{S}$ and each point $x\in T$ there exists an open neighbourhood $x\in V_x\subset X$ which is homeomorphic to the product
    $$V_x\simeq \Cone(L_x)\times (V_x\cap T)$$
    where $V_x\cap T$ is homeomorphic to a ball, and $L_x$ is called the {\bf link} of $T$ at $x$.
\item The link $L_x$ has a stratification indexed by the strata $S\in \mathcal{S}$ such that $T\subset \overline{S}$
     $$L_x=\bigcup _{T\subset \overline{S}}L_{x,S}$$
     It has the property that $S\cap V_x=L_{x,S}\times (0,1)\times (V_x\cap T)$. The strata $L_{x,S}$ (and hence also the intersection $S\cap V_x$) may have finitely many connected components. We call the stratum $L_{x,S}$ the {\bf link of $T$ in $S$}.
\end{itemize}
\end{definition}

\begin{para} In the Definition \ref{def-normal-str} it is clear that for any point $x\in T$ the open neighbourhoods $V_x$ with the mentioned properties form a fundamental system of open neighbourhoods of $x$.
\end{para}

\begin{para} Any Whitney stratified space has a normal structure \cite{McP}.
\end{para}

\begin{para} For many stratified spaces with normal structure, the link $L_x$ of a stratum $T$ at $x$ is independent (up to a homeomorphism) of the point $x\in T$.
\end{para}

\begin{lemma} \label{str-sp-is-ok} A stratified space $(X,\mathcal{S})$ with a normal structure is a \ref{enum:loc-sc}-\ref{enum:loc-sa}-\ref{enum:l_*-preserves-cons}-stratified
  space.
\end{lemma}

\begin{proof} Since each stratum is a manifold, it is clearly locally simply connected and locally acyclic. We only need to verify the \ref{enum:l_*-preserves-cons} condition.

Let $l:S\ra X$ be the embedding of a stratum, $F\in \Loc (S)$, and choose a stratum $T$ such that $T\subset \overline{S}$. We need to prove that the cohomology sheaves of the restriction of the complex $\Rd l_*F \in \D ^+(X)$ to $T$ are locally constant. Let $x\in T$ and let $V_x\subset X$ be an open neighbourhood as in the Definition \ref{def-normal-str}. The sheaf $\HH ^i(\Rd l_*F)\vert _{T}$ is associated to the presheaf
\begin{equation}\label{presh} V_x\cap T\mapsto \HH^i(V_x\cap S,F)
\end{equation}
We have $V_x\cap S=L_{x,S}\times (0,1)\times (V_x\cap T)$. Since $F\in \Loc (S)$ and the spaces $(0,1)$ and $V_x\cap T$ are contractible, we have
$$\HH ^i(V_x\cap S,F)=\HH ^i(L_{x,S},F)$$
therefore the presheaf \eqref{presh} is locally constant.
\end{proof}

\begin{para} So for a stratified space $(X,\mathcal{S})$ with a normal structure it makes sense to ask if the equivalent conditions of Theorem \ref{t:equivalent-conditions-for-equivalence-1} are satisfied. A partial answer is given by the following corollary.
\end{para}

\begin{corollary} \label{suf-cond-for-eq} Let $(X,\mathcal{S})$ be a stratified space with a normal structure. Let $l:S\ra X$ be the inclusion of a stratum. Then the following are equivalent:

(1) The morphism
\eqref{eq:equi-D+-sigma-1} in Theorem \ref{t:equivalent-conditions-for-equivalence-1}
\begin{equation}
      \label{eq:equi-D+-sigma-2}
      \sigma \colon \real \circ \Rd_\cs l_*
      \xRa{\eqref{eq:real-e_*-sigma-1}} \Rd l_* \circ
      \real
    \end{equation}
    is a 2-isomorphism.

    (2) Fix an injective object $I\in \Loc(S)$. Then for any stratum $T\subset \overline{S}$, any point $x\in T$, and any open neighbourhood $x\in V_x\subset X$ as in Definition \ref{def-normal-str} the sheaf
    $I\vert _{L_{x,S}}$ is acyclic for the functor
    $$\Gamma (L_{x,S},-)\colon \Sh (L_{x,S})\ra \Mod (R)$$
\end{corollary}

\begin{proof} $(1)\Rightarrow (2):$ Let $I\in \Loc (S)$ be injective. Then
$\Rd_{\cs}l_*(I)=l_*(I)$. Hence also $\Rd l_*(I)=l_*(I)$. Choose a stratum $T\subset \overline{S}$ and $x\in T$. In the proof of Lemma \ref{str-sp-is-ok} we showed that the stalk $\HH ^i(\Rd l_*(I))_x$ is isomorphic to the cohomology $\HH ^i(L_{x,S},I)$. Hence
$\HH ^i(L_{x,S},I)=0$ for $i>0$, i.e. $I\vert _{L_{x,S}}$ is acyclic for the functor $\Gamma (L_{x,S},-)$.

$(2)\Rightarrow (1):$ Let $A\in \D ^+(\Loc(S))$. Choose a quasi-isomorphism
$A\simeq I^\bullet$, where $I^\bullet$ is a bounded below complex of injective objects in $\Loc(S)$. By definition
$\Rd _\cs l_*( A)=l_*(I^\bullet)$. It suffices to show that an injective object $I\in \Loc (S)$ is acyclic for the functor
\begin{equation}\label{acycl-for-l*}
l_*:\Sh (S)\to \Sh (X)
\end{equation}
For this we again recall the proof of Lemma \ref{str-sp-is-ok}.

Choose a stratum $T\subset \overline{S}$ and $x\in T$. In the proof of Lemma \ref{str-sp-is-ok} we showed that the stalk $\HH ^i(\Rd l_*(I))_x$ is isomorphic to the cohomology $\HH ^i(L_{x,S},I)$. By our assumption the obiect $I$ is acyclic for the functor $\Gamma (L_{x,S},-)$. Hence
$\HH ^i(L_{x,S},I)=0$ for $i>0$. Therefore $I$ is acyclic for the functor \eqref{acycl-for-l*}.
\end{proof}

\begin{para} We want to describe a general class of examples where the equivalent conditions of Theorem \ref{t:equivalent-conditions-for-equivalence-1} are satisfied.
\end{para}

\begin{theorem}\label{our-original-example} Let $(X,\mathcal{S})$ be a stratified space with a normal structure. Assume that each stratum $S$ is a $K(\pi ,1)$ manifold, i.e. its universal covering space is contractible. Assume in addition that for any strata $S,T$, such that $T\subset \overline{S}$,  and any point $x\in T$, each connected component $L_{x,S,i}$ of the manifold $L_{x,S}$ is also a $K(\pi ,1)$-space.

(1) Then the equivalent conditions of Theorem  \ref{t:equivalent-conditions-for-equivalence-1} hold if for each pair of strata $T\subset \overline{S}$, each point $x\in T$ and each connected component
$L_{x,S,i}$ of the manifold $L_{x,S}$ the following holds:

(*) The restriction functor $r\colon \Mod (R\pi _1(S))\ra \Mod (R\pi _1(L_{x,S,i}))$ induced by the natural homomorphism $\pi _1(L_{x,S,i})\to \pi _1(S)$ takes injective objects to ones which are acyclic for the left exact functor of taking invariants
$$\Hom _{R\pi _1(L_{x,S,i})}(R,-)=(-)^{\pi _1(L_{x,S,i})}$$
where $R$ is the augmentation left $R\pi _1(L_{x,S,i})$-module.

(2) The condition (*)  holds in particular in any of the following cases:

(a) The homomorphism $\pi _1(L_{x,S,i})\to \pi _1(S)$ is injective.

(b) The kernel of the homomorphism $\pi _1(L_{x,S,i})\to \pi _1(S)$ is a finite subgroup whose order is prime to the characteristic of $R$.
\end{theorem}

\begin{proof} Since $(X,\mathcal{S})$ is a stratified space with a normal structure, by Lemma \ref{str-sp-is-ok} it is a \ref{enum:loc-sc}-\ref{enum:loc-sa}-\ref{enum:l_*-preserves-cons}-stratified
  space, so the assumptions of Theorem \ref{t:equivalent-conditions-for-equivalence-1} hold.

(1):
Notice that for each stratum $S\in \mathcal{S}$ its universal covering space $\tilde{S}$ is contractible, hence the realization functor
$$\real :\D ^+(\Loc (S))\ra \D ^+_{\Loc}(S)$$
is an equivalence (Theorem \ref{t:main-one-stratum-1}). So assuming that the condition $(*)$ holds, it suffices to prove that the morphism of functors
\begin{equation}
      \label{eq:equi-D+-sigma-3}
      \sigma \colon \real \circ \Rd_\cs l_*
      \xRa{\eqref{eq:real-e_*-sigma-1}} \Rd l_* \circ
      \real
    \end{equation}
    is a 2-isomorphism for any inclusion $l\colon S\ra X$ of a stratum.

Let $I\in \Loc (S)$ be an injective object. Choose a stratum $T\subset \overline{S}$, a point $x\in T$ and its neighbourhood $x\in V_x\subset X$ as in Definition \ref{def-normal-str}. By Corollary \ref{suf-cond-for-eq} it suffices to prove that the sheaf $I\vert _{L_{x,S}}$ is acyclic for the functor
$$\Gamma (L_{x,S},-)=\Hom (\ul{R}_{L_{x,S}},-) \colon \Sh (L_{x,S})\ra \Mod (R)$$
It suffices to check this for each connected component $L_{x,S,i}$ of $L_{x,S}$.

The natural functor $\D ^+(\Loc (L_{x,S,i}))\ra \D ^+_{\Loc}(\Sh (L_{x,S,i}))$ is an equivalence
(Theorem \ref{t:main-one-stratum-1}). So it is enough to prove that the sheaf $I\vert _{L_{x,S,i}}$ is acyclic for the functor
$$\Hom (\ul{R}_{L_{x,S,i}},-) \colon \Loc(L_{x,S,i})\ra \Mod (R)$$

Let $\beta : L_{x,S,i}\ra S$ be the embedding.
We have the commutative functorial diagram
  \begin{equation*}
    \xymatrix{
      {\Loc(S)} \ar[r]^{\beta ^*} \ar[d]^-{\sim} &
      {\Loc(L_{x,S,i})} \ar[d]^-{\sim}\\
      {\Mod(R\pi _1(S))} \ar[r]^-{r} &
      {\Mod(R \pi _1(L_{x,S,i}))}
    }
  \end{equation*}
where the vertical functors are equivalences (Proposition~\ref{p:fiber-functor-1}). Now the assumption (*) implies that the sheaf $I\vert _{L_{x,S,i}}=\beta ^*(I)$ is acyclic for the functor
$$\Hom (\ul{R}_{L_{x,S,i}},-) \colon \Loc(L_{x,S,i})\ra \Mod (R)$$
This proves (1).

(2)(a): Suppose that the homomorphism $\pi _1(L_{x,S,i})\to \pi _1(S)$ is injective. Then the group ring $R\pi _1(S)$ is a free $R\pi _1(L_{x,S,i})$-module. So the restriction of scalars functor $r\colon \Mod (R\pi _1(S))\ra \Mod (R\pi _1(L_{x,S,i}))$ preserves injectives, as it is the right adjoint to the exact functor $R\pi _1(S)\otimes _{R\pi _1(L_{x,S,i})}(-)$.

(2)(b): Let $K$ be the kernel of the homomorphism
$$H:= \pi_1(L_{x,S,i}) \ra G:= \pi_1(S)$$
and assume that $K$ is finite and its order is prime to the characteristic of the ring $R$.  Let $I\in \Mod (RG)$ be an injective object and $r(I)\in \Mod (RH)$ the corresponding $RH$-module. We need to show that $r(I)$ is acyclic for the functor $\Hom _{RH}(R,-)=(-)^H$ of taking $H$-invariants.
By part (a) we may assume that the homomorphism $H\to G$ is surjective.
Let
\begin{equation}\label{before-k-inv}
0\ra r(I)\ra J^0\ra J^1\ra \cdots
\end{equation}
be an injective resolution in the category $\Mod (RH)$. For any $M\in \Mod (RH)$ we have
$M^H=(M^K)^G$. First notice that the complex
\begin{equation}\label{after-k-inv}
0\ra r(I)^K\ra (J^0)^K\ra (J^1)^K\ra \cdots
\end{equation}
is exact. Indeed, since $K$ is a finite group whose order is prime to $char R$, the surjection of $RK$-modules $RK\to R$ has a splitting (these are free $R$-modules), so the functor $(-)^K=\Hom _{RK}(R,-)\colon \Mod (RK)\ra \Mod (R)$ is exact. Since $r(I)^K=r(I)$, the exact complex \eqref{after-k-inv} is a resolution of the injective $RG$-module $I$ and so it remains exact after taking $G$-invariants.
\end{proof}

\subsection{Examples}

\begin{proposition}\label{toric-var}
Let $X$ be a complex normal separated toric variety.
Consider $X$ as an analytic space with classical topology and the natural stratification $\mathcal{S}$ by torus orbits. The stratified space $(X,\mathcal{S})$ has a normal structure (Definition \ref{def-normal-str}), the conditions of Theorem \ref{our-original-example} are satisfied and hence the functor 
\eqref{eq:equi-D+-X} is an equivalence.
\end{proposition}

\begin{proof} It is known (see for example \cite{CLS}) that a normal separated toric variety comes from a fan. Let $T$ be the corresponding complex torus. Choose an orbit $O\subset X$. Let $St(O)\subset X$ be the {\it star} of $O$, i.e. it is the union of orbits $S\subset X$ such that $O\subset \overline{S}$. Since $X$ comes from a fan, $St(O)$ is an open affine subvariety of $X$ and $O$ is a minimal orbit in $St(O)$. Moreover, there exist subtori $T',T''\subset T$ such that \begin{itemize}
\item $T'$ is the stabilizer of the orbit $O$, and $O\simeq T''$.
\item There exists an affine $T'$-toric variety $X'$ with a fixed point and an isomorphism of $T$-toric varieties $X\simeq X'\times O$.
\end{itemize}
To examine the normal structure of $X$ along the orbit $O$
we first consider the affine $T'$-toric variety $X'$ with a fixed $p$ as above.
Then there exists a 1-parameter subgroup $\lambda _t\subset T'$ which contracts $X'$ to $p$ as $t\to \infty$. This means that if $B_\epsilon \subset X'$ is a small sphere centered at $p$, then for any $T'$-orbit $S'\subset T'$ the intersection $B_\epsilon \cap S'$ is homotopy equivalent to $S'$, and $B_\epsilon \cap S'$ is the link of $p$ in $S'$. In particular, this link $B_\epsilon \cap S'$ is connected, is a $K(\pi ,1)$-space and the corresponding map of fundamental groups
$$\pi _1 (B_\epsilon \cap S')\to \pi _1(S')$$
is an isomorphism.

Let now $S\subset St(O)$ be a $T$-orbit. Then $S=O\times S'$ for a $T'$-orbit $S'\subset X'$ and the link of $O$ in $S$ is equal to the link
$B_\epsilon \cap S'$ of $p$ in $S'$ as above. Hence this link is a $K(\pi ,1)$-space and the corresponding map of the fundamental groups
$$\pi _1 (B_\epsilon \cap S')\to \pi _1(S)$$
is injective. This proves the proposition.
\end{proof}


\begin{remark}
  For $X=\bP^1\bC$ with the Bruhat stratification $X=\bC \sqcup \{pt\}$ the functor   \eqref{eq:equi-D+-X} is not an equivalence. Here the strata are contractible, but the link is homotopic to $S^1$, so the conditions of Theorem \ref{our-original-example} fail.
\end{remark}

\section{Versions of the main theorem}\label{sec-versions}

Assume that $(X,\mathcal{S})$ is a \ref{enum:loc-sc}-\ref{enum:loc-sa}-\ref{enum:l_*-preserves-cons}-stratified
  space and we are in the situation of Theorem \ref{t:equivalent-conditions-for-equivalence-1}
  i.e. the natural functor
    \begin{equation}
      \label{eq:equi-D+-X-1}
      \real \colon \D^+(\Cons(X, \mathcal{S})) \ra
      \D_\mathcal{S}^+(X)
    \end{equation}
    is an equivalence.

\begin{definition} \label{ft}  Consider the full subcategory $\Cons _{ft}(X, \mathcal{S})\subset \Cons(X, \mathcal{S})$ consisting of constructible sheaves of {\bf finite type}, i.e. sheaves $F$ such that the stalk $F_x$ is a finitely generated $R$-module for any $x\in X$.
\end{definition}

\begin{lemma} Assume that the ring $R$ is left Noetherian. Then $\Cons _{ft}(X, \mathcal{S})$ is an abelian subcategory of $\Cons(X, \mathcal{S})$ closed under extensions and direct summands.
\end{lemma}

\begin{proof} Given a morphism $\gamma :F\to G$ in $\Cons _{ft}(X, \mathcal{S})$ the kernel and the cokernel of the morphism of stalks $\gamma _x :F_x\to G_x$ are finitely generated $R$-modules, hence $ker (\gamma ),coker (\gamma )\in \Cons _{ft}(X, \mathcal{S})$. Similarly, it is clear that $\Cons _{ft}(X, \mathcal{S})$ is closed under extensions in $\Cons(X, \mathcal{S})$ and direct summands.
\end{proof}

\begin{para} A constructible sheaf $F\in \Cons(X, \mathcal{S})$ is of finite type if and only if for any stratum $S\in \mathcal{S}$ the $R\pi _1(S)$ module corresponding to the local system $F\vert _S$ is finitely generated as an $R$-module.
\end{para}

\begin{definition} Assume that $R$ is left Noetherian. Denote by $\D ^+_{ft}(\Cons (X,\mathcal{S}))$ the thick triangulated subcategory of $\D ^+(\Cons (X,\mathcal{S}))$ with cohomology sheaves in $\Cons _{ft}(X, \mathcal{S})$.
\end{definition}

The following corollary is obvious.

\begin{corollary} Assume that $R$ is left Noetherian and the functor
\eqref{eq:equi-D+-X-1} is an equivalence. Then it induces the equivalence
\begin{equation}\label{ind-equiv}
\real \colon \D ^+_{ft}(\Cons (X,\mathcal{S}))\ra \D ^+_{\Cons _{ft}(X,\mathcal{S})}(X)
\end{equation}
\end{corollary}

One then may ask the following natural question.

\medskip
\noindent{\bf Question (*).} Suppose that $R$ is left Noetherian. When is the natural functor
\begin{equation} \label{real-bounded-ft}
\D ^b(\Cons _{ft}(X,\mathcal{S}))\ra \D ^b_{\Cons _{ft}(X,\mathcal{S})}(X)
\end{equation}
an equivalence?
This is an interesting and subtle question, that we cannot answer in general. The problem is that even if the functor \eqref{eq:equi-D+-X-1} is an equivalence, we do not know when the functor
\begin{equation}\label{basic-eq-q}
\D ^b(\Cons _{ft}(X,\mathcal{S}))\ra \D ^b_{ft}(\Cons (X,\mathcal{S}))
\end{equation}
is an equivalence. In fact it is not clear even in the case of one stratum in general. Below we give a positive answer to the question (*) in some special cases. We start with a general result.

\begin{proposition}\label{when-canon-f-is-eq} Let $\mathcal{A}$ be an abelian category and $\mathcal{B}\subset \mathcal{A}$ its strictly full abelian subcategory
closed under extensions. Assume that the following condition holds:
For any $M\in \mathcal{A}$ there exists a subobject $N\subset M$ such that

(i) $M/N\in \mathcal{B}$;

(ii) $N$ has no nonzero subobjects which belong to $\mathcal{B}$.

\noindent Then the canonical functor $\Psi :\D ^b (\mathcal{B})\to \D ^b_{\mathcal{B}}(\mathcal{A})$ is an equivalence.
\end{proposition}

\begin{proof} \underline{$\Psi$ is essentially surjective:} Let
$$A^\bullet :=0\ra A^i\stackrel{d^i}{\ra} A^{i+1}\ra \cdots
\stackrel{d^{n-1}}{\ra} A^n\ra 0$$
be a complex in $\D ^b_{\mathcal{B}}(\mathcal{A})$. Let $t$ be the lowest index such that $A^t\notin \mathcal{B}$. Choose a subobject $P\subset A^t$ such that $A^t/P\in \mathcal{B}$ and $P$ has no nonzero subobjects from $\mathcal{B}$. We claim that $P\cap ker(d^t)=0$. Indeed, since $H^t(A\bullet ), A^{t-1}\in \mathcal{B}$, we have $ker (d^t)\in \mathcal{B}$, so $P\cap ker(d^t)=0$. Therefore $A\bullet $ contains the acyclic complex
$$\tilde{P}:=P\stackrel{\sim}{\ra} d^t(P)$$
and the components of $A^\bullet /\tilde{P}$ with index $\leq t$ belong to $\mathcal{B}$. proceeding by induction we eventually find a quasi-isomorphism $A^\bullet \ra B^\bullet$, where $B^\bullet \in \C^b(\mathcal{B})$.

\underline{$\Psi$ is full and faithful:} Let $C^\bullet ,D^\bullet \in \D ^b(\mathcal{B})$. A morphism $\Psi (D^\bullet)\to \Psi (C^\bullet)$ is represented by a diagram of morphisms of complexes in $\C^b(\mathcal{A})$
$$D^\bullet \ra A^\bullet \stackrel{s}{\la}C^\bullet$$
where $s$ is a quasi-isomorphism. As we showed above, there is a complex $B^\bullet \in \C^b(\mathcal{B})$ and a morphism of complexes $A^\bullet \ra B^\bullet$ which is a quasi-isomorphism. This proves that $\Psi$ is full and faithful.
\end{proof}

\begin{para} Let $A$ be a Noetherian commutative algebra over a field. Let $J\subset A$ be an ideal. Put $B:=A/J$ and consider the abelian categories $\mathcal{A}:=\Mod (A)$ and its abelian subcategory
$\mathcal{A}_{fg}\subset \mathcal{A}$ of finitely generated modules.
Similarly we have the abelian subcategory $\mathcal{B}_{fg}\subset \mathcal{A}$ of finitely generated $B$-modules.
\end{para}

\begin{corollary} \label{cor-for-can-ft} The canonical functor $\D ^b(\mathcal{B}_{fg})\to
\D ^b_{\mathcal{B}_{fg}}(\mathcal{A})$ is an equivalence.
\end{corollary}

\begin{proof} First recall the well-known fact that the canonical functor
$$\D ^b(\mathcal{A}_{fg})\ra \D ^b_{\mathcal{A}_{fg}}(\mathcal {A}).$$
is an equivalence. Clearly it induces the equivalence
$$\D ^b_{\mathcal{B}_{fg}}(\mathcal{A}_{fg})\ra \D ^b_{\mathcal{B}_{fg}}(\mathcal {A}).$$
So it suffices to prove that the canonical functor
$$\Psi :\D ^b(\mathcal{B}_{fg})\to \D ^b_{\mathcal{B}_{fg}}(\mathcal{A}_{fg})$$
is an equivalence. This follows from Proposition \ref{when-canon-f-is-eq}. Indeed, for any finitely generated $A$-module $M$ its $J$-torsion submodule $M^J$ is finitely generated, hence for $n>>0$ we have $J^nM^J=0$ and so can take $N:=J^nM$.
\end{proof}

\begin{remark} Corollary \ref{cor-for-can-ft} can be extended to some noncommutative Noetherian rings (Larsen-Lunts "Lie algebra cohomology").
\end{remark}

\begin{para} A version of Corollary \ref{cor-for-can-ft} is the following.
Let $\mathcal{A}$ and $\mathcal{A}_{fg}$ be as in Corollary \ref{cor-for-can-ft}. Let $\mathcal{C}\subset \mathcal{A}_{fg}$ be the abelian subcategory of modules with zero-dimensional support.
\end{para}

\begin{corollary} The canonical functor $\D ^b(\mathcal{C})\to
\D ^b_{\mathcal{C}}(\mathcal{A})$ is an equivalence.
\end{corollary}

\begin{proof} Similar to that of Corollary \ref{cor-for-can-ft}.
\end{proof}

\subsection{When is the functor $\D ^b(\Cons _{ft}(X,\mathcal{S}))\ra \D ^b_{ft}(\Cons (X,\mathcal{S}))$ an equivalence?}

\begin{definition} \label{cond!} Let $A$ be an algebra over a field $k$. An $A$-module $M$ is of {\bf finite type} (resp. {\bf ind-finite type}) if it is finite dimensional over $k$ (resp. if it is a union of modules of finite type). Let
$$\Mod _{ft}(A)\subset \Mod _{ift}(A)\subset \Mod (A)$$
be the corresponding full subcategories.
We say that the algebra $A$ satisfies the condition (!) if any $M\in \Mod _{ift} (A)$ is a submodule of an injective $A$-module $I$ which belongs to
$\Mod _{ift}(A)$.
\end{definition}

\begin{theorem} \label{when eq} Assume that the coefficient ring $R$ is a field $k$.

Let $(X,\mathcal{S})$ be a stratified space with normal structure (Definition \ref{def-normal-str}). Assume that for any stratum $S\in \mathcal{S}$ the $k$-algebra $k\pi _1(S)$ satisfies the condition (!) of Definition \ref{cond!}. Then the natural functor
\begin{equation}\label{f-between-const}
\D ^b(\Cons _{ft}(X,\mathcal{S}))\ra \D ^b_{ft}(\Cons (X,\mathcal{S}))
\end{equation}
is an equivalence. (The objects in $\Cons _{ft}(X,\mathcal{S})$ are constructible sheaves with $k$-finite dimensional stalks.)
\end{theorem}

\begin{proof} \underline{Step 1.} Note that since $(X,\mathcal{S})$ is a stratified space with a normal structure the following holds:

(a) For any (connected) stratum $S$ there is an equivalence (Proposition \ref{p:fiber-functor-2})
\begin{equation}\label{recal-equiv} \Loc (S)\simeq \Mod (k\pi _1(S))
\end{equation}
 Denote by
$$\Loc _{ft}(S)\subset \Loc _{ift}(S)\subset \Loc(S)$$
the subcategories corresponding under this equivalence to the similar subcategories of $\Mod (k\pi _1(S))$. Let also
$$\Cons _{ft}(X)\subset \Cons _{ift}(X)\subset \Cons(X)$$
be the full subcategories consisting of sheaves whose restriction to any stratum $S$ belongs to the corresponding subcategory of $\Loc (S)$.

(b) For any embedding of a stratum $l\colon S\ra X$ the functor $l_*:\Loc (S)\ra \Cons (X)$ (Lemma \ref{str-sp-is-ok}) restricts to the functors
$$l_*:\Loc _{ft}(S)\ra \Cons _{ft}(X),\quad \text{and}\quad
l_*:\Loc _{ift}(S)\ra \Cons _{ift}(X)$$
Indeed, this follows from the proof of Lemma \ref{str-sp-is-ok} and the fact that any link  $L_{x,S}$ has finitely many connected components.

\medskip
\noindent\underline{Step 2.} Any $G\in \Cons _{ift}(X)$ is a union of its subsheaves which belong to $\Cons _{ft}(X)$. Indeed, let $W\subset X$ be the support of $G$ (so $W$ is a union of strata), and let $j:U\hookrightarrow W\hookleftarrow Z:i$ be the embedding of an open (in $W$) stratum $U$ and its closed complement $Z$. We have the exact sequence in $\Cons _{ift}(X)$\begin{equation}\label{short exact in const}
0\to j_!j^*G\to G\to i_*i^*G\to 0
\end{equation}
By the equivalence \eqref{recal-equiv} the sheaf $j^*G\in  \Loc _{ift}(U)$ is a union of its subsheaves from $\Loc _{ft}(U)$. Hence the sheaf $j_!j^*G$ is a union of its subsheaves from $\Cons _{it}(X)$. Hence we may replace the sheaf $G$ by the sheaf $i_*i^*G$ whose support is smaller, and conclude by induction on the size of the support.

\medskip
\noindent{\underline{Step 3.}} For any $G\in \Cons _{ift}(X)$ there exists an embedding $G\hookrightarrow I$, where $I\in \Cons (X)$ is injective and belongs to $\Cons _{ift}(X)$. Indeed, let $l\colon S\ra X$ be the inclusion of a stratum. Then the locally constant sheaf $G\vert _S$ belongs to $\Loc _{ift}(S)$. Hence by the equivalence \eqref{recal-equiv} and our assumption (!) on the algebra $k\pi _1(S)$, there exists an embedding $G\vert _S\hookrightarrow I_S$, where $I_S\in \Loc (S)$ is injective and belongs to $\Loc _{ift}(S)$. Notice that the sheaf $l_*I_S\in \Cons (X)$ is injective and belongs to $\Cons _{ift}(X)$. It remains to take $I:=\bigoplus _{S}l_*I_S$ with the diagonal embedding
$$G\hookrightarrow \bigoplus _Sl_*(G\vert _{S})\hookrightarrow \bigoplus _Sl_*I_S.$$

\medskip
\noindent{\underline{Step 4.}} For any $G^\bullet \in \C^b(Cons _{ift})$ there exists a quasi-isomorphism
$$G^\bullet \ra J^\bullet \in \C^+(\Cons _{ift}(X))$$
where the complex $J^\bullet$ consists of sheaves that are injective in $\Cons (X)$. Indeed, for each member $G^i$ of the complex $G^\bullet$ we may use repeatedly Step 2 to construct a resolution $G^i\to I^{\bullet ,i}$, consisting of sheaves in $\Cons _{ift}(X)$ which are injective in $\Cons (X)$. Then the complex $J^\bullet$ may be constructed from the complexes $I^{\bullet ,i}$ in a standard way.

\medskip
\noindent{\underline{Step 5.}} Assume that $J^\bullet \in \C^+(\Cons _{ift}(X))$ consists of injective constructible sheaves, has bounded cohomology, and each $\HH^i(J^\bullet )\in \Cons _{ft}(X)$. Then there exists a bounded subcomplex $K^\bullet \subset J^\bullet$ such that

(i) $K^\bullet \in \C^b(\Cons _{ft}(X))$;

(ii) the embedding $K^\bullet \hookrightarrow J^\bullet$ is a quasi-isomorphism.

We construct the subcomplex $K^\bullet $ by descending induction starting with the highest index $n$ such that $\HH ^n(J^\bullet )\neq 0$. Since $J^n\in \Cons _{ift}(X)$ and $\HH ^n(J^\bullet )\in \Cons _{ft}(X)$ it is easy to see using Step 2, that we can find a subsheaf $K^n\subset Z^n(J^\bullet)$ which projects onto $\HH ^n(J^\bullet )$ and belongs to $\Cons _{ft}(X)$.
$$\cdots \ra J^{n-1}\stackrel{d^{n-1}}{\ra} J^n\ra \cdots$$
Next, let $K^{n-1}\subset J^{n-1}$ be a subsheaf from $\Cons _{ft}(X)$ with the 2 properties: (1) $d^{n-1}$ maps $K^{n-1}$ surjectively onto $d^{n-1}(J^{n-1})\cap K^n$, (2) $K^{n-1}\cap Z^{n-1}(J^\bullet)$ projects surjectively onto $\HH ^{n-1}(J^\bullet)$. Then $K^{n-2}\subset J^{n-2}$ is chosen with the above properties replacing $K^n$ by $K^{n-1}$, and so on.

\medskip
\noindent{\underline{Step 6.}} Notice that in Step 4 the required subcomplex $K^\bullet \subset J^\bullet $ can be chosen to contain any given subcomplex $F^\bullet \subset J^\bullet$ from $\C^b(\Cons _{ft})$. Indeed, first replace $J^\bullet$ by $J^\bullet /F^\bullet$ and find $\overline{K}^\bullet \subset J^\bullet /F^\bullet$ with the required properties. Then take $K^\bullet$ to be the preimage of $\overline{K}^\bullet$.

\medskip
\noindent{\underline{Step 7.}} Both triangulated categories $\D ^b(\Cons _{ft}(X))$ and $D^b_{ft}(\Cons (X))$ are generated by $\Cons _{ft}(X)$. So it suffices to prove that the morphism spaces $\Hom (F,G[m])$ in the two categories are isomorphic for any $F,G\in \Cons _{ft}(X)$ and any integer $m$. Choose a resolution $s:G[m]\ra J ^\bullet$ as in Step 4. A morphism $\alpha :F\to G[m]$ in $\D ^b_{ft}(Cons (X))$ is represented by a diagram of morphisms of complexes
$$F\stackrel{f}{\ra}J^\bullet \stackrel{s}{\la}G[m]$$
Choose a subcomplex $K^\bullet \subset J^\bullet $ as in Step 5, which also contains the $f(F)$ and $s(G[m])$ (Step 6). Then $s:G[m]\to K^\bullet$ is a quasi-isomorphism and the morphism $\alpha$ is also represented by the diagram
$$F\stackrel{f}{\ra}K^\bullet \stackrel{s}{\la}G[m]$$
This shows that $\Hom (F,G[m])$ is the same in categories $\D ^b(\Cons _{ft}(X))$ and $D^b_{ft}(\Cons (X))$ and proves Theorem \ref{when eq}.
\end{proof}

\begin{corollary} \label{cor for ab fund grp} Assume that the coefficient ring $R$ is a field $k$.

Let $(X,\mathcal{S})$ be a stratified space with normal structure. Assume that for any stratum $S\in \mathcal{S}$ the fundamental group $\pi _1(S)$ is finitely generated abelian. Then the natural functor
\begin{equation}\label{f-between-const}
\D ^b(\Cons _{ft}(X,\mathcal{S}))\ra \D ^b_{ft}(\Cons (X,\mathcal{S}))
\end{equation}
is an equivalence.
\end{corollary}

\begin{proof} By Theorem \ref{when eq} it suffices to prove than any finitely generated commutative $k$-algebra $A$ satisfies the condition (!) of Definition \ref{cond!}. This follows from the structure  of indecomposable injectives in $\Mod (A)$ \cite{Hart}.
\end{proof}

\begin{corollary}\label{cor vers toric var} Let $R$ be a field. Let $(X,\mathcal{S})$ be a complex toric variety with the orbit stratification. Then the functor \ref{real-bounded-ft} is an equivalence.
\end{corollary}

\begin{proof} Since all strata are complex tori, their fundamental groups are finitely generated abelian. Hence by Corollary \ref{cor for ab fund grp} the functor
$$\D ^b(\Cons _{ft}(X,\mathcal{S}))\ra \D ^b_{ft}(\Cons (X,\mathcal{S}))$$
is an equivalence. But by
Proposition \ref{toric-var} the functor
$$\real \colon \D ^b_{ft}(\Cons (X,\mathcal{S}))\ra \D ^b_{\Cons _{ft}(X,\mathcal{S})}(X)$$
is also an equivalence.
\end{proof}

\def\cprime{$'$} \def\cprime{$'$} \def\cprime{$'$} \def\cprime{$'$}
  \def\Dbar{\leavevmode\lower.6ex\hbox to 0pt{\hskip-.23ex \accent"16\hss}D}
  \def\cprime{$'$} \def\cprime{$'$}

\end{document}